\documentclass{amsart}

\usepackage[english]{babel}
\usepackage[utf8]{inputenc}

\usepackage[left=3cm,top=2.5cm,bottom=2.5cm,right=3cm]{geometry}	
\usepackage[dvipsnames]{xcolor}				
\usepackage{enumitem}
\usepackage{csquotes}	

\usepackage{todonotes}	

\usepackage{hyperref}			
\hypersetup{
	colorlinks=true,%
	linktocpage=true,%
	pdfstartpage=1,%
	pdfstartview=FitV,%
	breaklinks=true,%
	pdfpagemode=UseNone,%
	pageanchor=true,%
	pdfpagemode=UseOutlines,%
	plainpages=false,%
	bookmarksnumbered,%
	bookmarksopen=true,%
	bookmarksopenlevel=1,%
	hypertexnames=true,%
	pdfhighlight=/O,%
	urlcolor=BrickRed,%
	linkcolor=RoyalBlue,%
	citecolor=ForestGreen
}

\usepackage[
	style=alphabetic,%
	giveninits,%
	sorting=nyt,%
	maxnames=10,
	maxalphanames=6,%
	backend=biber
]{biblatex}
\usepackage{amsmath,amssymb}		
\usepackage{mathtools,bm,braket,faktor,stmaryrd,nicematrix}	

\newcommand{\bbraket}[1]{\llbracket #1 \rrbracket}

\newcommand{\Z}{\mathbb{Z}}

\newcommand{\C}{\mathbb{C}}
\renewcommand{\P}{\mathbb{P}}

\newcommand{\iu}{\mathrm{i}}
\newcommand{\Id}{\mathrm{Id}}

\newcommand{\mc}[1]{\mathcal{#1}}

\AtBeginDocument{\renewcommand{\hbar}{\hslash}}
\newcommand{\bigO}{\mathord{\mathrm{O}}}

\newcommand{\Mbar}{\overline{\mathcal{M}}}

\DeclareMathOperator*{\Res}{Res}

\DeclareMathOperator{\sgn}{sgn}

\usepackage{thmtools}
\usepackage[noabbrev]{cleveref}

\theoremstyle{plain}
\newtheorem{theorem}{Theorem}[section]
\newtheorem{proposition}[theorem]{Proposition}
\newtheorem{lemma}[theorem]{Lemma}
\newtheorem{corollary}[theorem]{Corollary}

\newtheorem{introthm}{Theorem}

\theoremstyle{definition}
\newtheorem{definition}[theorem]{Definition}
\newtheorem{remark}[theorem]{Remark}
\newtheorem{example}[theorem]{Example}

\crefname{lemma}{lemma}{lemmata}
\Crefname{lemma}{Lemma}{Lemmata}
\crefname{subsection}{subsection}{subsections}
\Crefname{subsection}{Subsection}{Subsections}
\crefname{conjecture}{conjecture}{conjectures}
\Crefname{conjecture}{Conjecture}{Conjectures}

\title[Theta classes]{Theta classes: generalized topological recursion, integrability and $\mathcal{W}$-constraints}

\author{V. Bouchard}%
\address[V.B.]{%
	Department of Mathematical {\&} Statistical Sciences, University of Alberta, 632 CAB, Edmonton, Alberta, Canada, T6G 2G1%
}%
\email{vincent.bouchard@ualberta.ca}
\author{N. K. Chidambaram}%
\address[N.K.C.]{
	School of Mathematics, University of Edinburgh, James Clerk Maxwell Building, Peter Guthrie Tait Rd, Edinburgh EH9 3FD, U.K. %
}
\curraddr{%
	 Departamento de Matemáticas Fundamentales, UNED, Calle de Juan del Rosal 10, 28040 Madrid, Spain%
}%
\email{nitin.chidambaram@mat.uned.es}
\author{A. Giacchetto}%
\address[A.G.]{%
	ETH Z{\"u}rich, Departement Mathematik, Rämistrasse 101, 8092 Z{\"u}rich, Switzerland%
}%
\email{alessandro.giacchetto@math.ethz.ch}
\author{S. Shadrin}%
\address[S.S.]{%
	Korteweg--de Vries Institute for Mathematics, University of Amsterdam, Postbus 94248, 1090GE Amsterdam, The Netherlands%
}%
\email{s.shadrin@uva.nl}
\thanks{
	V.B. is supported by the Natural Sciences and Engineering Research Council of Canada. %
	N.K.C. is supported by the ERC Starting Grant 948885, the Royal Society University Research Fellowship, and the Ramón y Cajal fellowship RYC2023-042878-I, funded by MCIN/AEI/\allowbreak10.13039/\allowbreak501100011033 and by the European Social Fund Plus (FSE+). %
	A.G. is supported by an ETH Fellowship (22-2~FEL-003) and a Hermann-Weyl-Instructorship from the Forschungsinstitut für Mathematik at ETH Zürich. %
	S.S. is supported by the Netherlands Organization for Scientific Research.%
	We thank the anonymous referee for a careful reading of the paper and suggestions for improvement.
}

\begin{document}
	
\begin{abstract}
	We study the intersection theory of the $\Theta^{r,s}$-classes, where $r \geq 2$ and $1 \le s \le r-1$, which are cohomological field theories obtained as the top degrees of Chiodo classes. We show that the recently introduced generalized topological recursion on the $(r,s)$ spectral curves computes the descendant integrals of the $\Theta^{r,s}$-classes. As a consequence, we deduce that the descendant potential of the $\Theta^{r,s}$-classes is a tau function of the $r$-KdV hierarchy, generalizing the Brézin--Gross--Witten tau function (the special case $r=2$, $s=1$). We also explicitly compute the $\mathcal{W}$-constraints satisfied by the descendant potential, obtained as differential representations of the $\mathcal{W}(\mathfrak{gl}_r)$-algebra at self-dual level. This work extends previously known results on the Witten $r$-spin class, the $r$-spin $\Theta$-classes (the case $s=r-1$), and the Norbury $\Theta$-classes (the special case $r=2$, $s=1$).
\end{abstract} 

\maketitle

\section{Introduction}
\label{sec:introduction}

Topological recursion~\cite{EO07} takes as input a spectral curve, which consists of a Riemann surface $\Sigma$, two meromorphic functions $x$ and $y$ on $\Sigma$, and a symmetric bi-differential $B$ on $\Sigma^2$, subject to some particular conditions. Via an explicit recursive procedure, it outputs a collection of symmetric meromorphic differentials $\set{ \omega_{g,n} }$ called ``correlators'' and defined for all $g \in \mathbb{Z}_{\geq 0}$, $n \in \mathbb{Z}_{>0}$, and $2g-2+n>0$. The standard assumption on $x$ and $y$ is that all zeros of $dx$ are simple and $dy$ does not vanish at these points~\cite{EO07}. However, this assumption can be relaxed: one can study spectral curves such that $dx$ has higher order zeros. This is done in a variety of papers, first by Bouchard--Eynard~\cite{BE13,BHLMR14}, and then further developed in~\cite{BBCCN24,BKS24,BBCKS25}. In this generalized context, the main building blocks of spectral curves are the \textit{$(r,s)$ spectral curves}, which consist of
\begin{equation}\label{eq:rscurve}
	\Sigma = \mathbb{P}^1,
	\qquad\qquad
	x = z^r,
	\qquad\qquad
	y = z^{s-r},
	\qquad\qquad
	B = \frac{dz_1 dz_2}{(z_1-z_2)^2},
\end{equation}
with $r \geq 2$ and $s \in [r+1]$ with $r \equiv \pm 1 \pmod{s}$. Here and throughout, for a positive integer $m$, we write $[m]=\set{1,\ldots,m}$. Throughout this paper, we use the convention that $\omega_{0,1} = -y dx$. The $(r,s)$ spectral curves are the building blocks for all spectral curves since they control the local behavior of general spectral curves near the zeros of $dx$.

One of the main areas of application of the theory of topological recursion is its connection to the intersection theory of the moduli spaces of curves~\cite{Eyn14} and, more specifically in some cases, to cohomological field theories~\cite{DOSS14}. In order to systematically develop this connection, one of the key questions is to identify the cohomology classes on the moduli space of curves whose intersection with $\psi$-classes (also known as descendant integrals) is controlled by the expansions of the correlators for the basic building blocks, the $(r,s)$ spectral curves. 

In the case when $s=r+1$, the correlators calculate descendant integrals for the $r$-spin Witten class, see~\cite{BE17,DNOPS19,CCGG24}. When $s=r-1$, the correlators calculate descendant integrals for the so-called $r$-spin $\Theta$-classes, which are defined as the top degrees of certain Chiodo classes~\cite{CGG25} (in the special case $r=2, s=1$, it reproduces the Norbury $\Theta$-classes~\cite{Nor23}). However, for the other cases with $r \geq 3$, $s = 2,\ldots, r-2$ and $r \equiv \pm 1 \pmod{s}$, the enumerative meaning of the correlators is an open question~\cite{BBCCN24,BKS24,CGG25}.

There is an alternative approach to the definition of topological recursion that comes from the idea that the correlators should be compatible with the universal $x-y$ swap formula~\cite{ABDKS24-logTR,ABDKS25-xy}, as well as with a more general group of symplectic transformations~\cite{ABDKS24-sym,BDKS24} (see also \cite{Hoc23,Hoc24}). This approach was further developed and related to the Bouchard--Eynard recursion in~\cite{ABDKS25-gTR}, where it was dubbed \emph{generalized topological recursion}. In this framework, the open question above is naturally replaced by a different one: what is the geometric meaning of the expansion of the correlators obtained by the generalized topological recursion on the $(r,s)$ spectral curves? Note that this alternative question does not require any modularity constraint on $r$ and $s$: generalized topological recursion applies to all $(r,s)$ spectral curves with $r \geq 2$ and $s \in [r-1]$. 

\subsection*{Generalized topological recursion}
The first goal of this paper is to answer this alternative question. We prove, in \cref{thm:WgnTheta} and \cref{thm:gTR}, that the correlators produced by generalized topological recursion on all $(r,s)$ spectral curves with $r \geq 2$ and $s \in [r-1]$ calculate descendant integrals of the so-called $\Theta^{r,s}$-classes. The $\Theta^{r,s}$-class (defined precisely in \cref{def:theta}) is the top degree piece of the Chiodo class $ C_{g,n}^{r,s}$, which in turn is defined via the moduli space \smash{$\Mbar^{r,s-r}_{g,a}$} of twisted spin curves with primary fields $a \in \mathbb Z^n$.

\begin{introthm}[{Generalized topological recursion for $\Theta^{r,s}$}]
	The correlators $\omega_{g,n}$ produced by generalized topological recursion on the $(r,s)$ spectral curve calculate the descendant integrals of the $\Theta^{r,s}$-classes. More precisely, for $g \ge 0$, $n \ge 1$ such that $2g-2+n>0$, we have
	\begin{equation}
		\omega_{g,n}(z_1,\ldots,z_n)
		=
		\left(\frac{1}{r}  \right)^{2g-2+n}
		\sum_{\substack{ k_1,\dots,k_n \ge 0 \\ a_1,\dots,a_n \in [r-1]}}
		\int_{\Mbar_{g,n}}
		\Theta_{g,n}^{r,s}(a_1,\ldots,a_n)
		\prod_{i=1}^n \psi_i^{k_i} \, d \xi_{k_i,a_i}(z_i),
	\end{equation}
	where $d \xi_{k,a}(z) = (rk+a)!^{(r)} \frac{dz}{z^{rk+a+1}}$.
\end{introthm}

This is a generalization of the $s=r-1$ result of~\cite{CGG25}, but  our approach is quite different from \textit{loc. cit.} In fact, it is interesting to note that for all $s \neq r-1$, the correlators produced by generalized topological recursion do not coincide with the correlators produced by the Bouchard--Eynard recursion. It thus remains a mystery what the geometric meaning of the latter is.

\subsection*{Integrability}
Our second result concerns integrability of the descendant potential $Z^{r,s}$ for the $\Theta^{r,s}$-classes, see \cref{eq:tau} for the precise definition. A direct consequence of the generalized topological recursion framework of~\cite{ABDKS25-gTR} is that the descendant potential is a KP tau function. In fact, we show in \cref{sec:integrability} that $Z^{r,s}$ is a tau function for the $r$-KdV hierarchy.

\begin{introthm}[Integrability]
	The descendant potential $Z^{r,s}(\bm{x};\hbar)$ of the $\Theta^{r,s}$-class is a tau function of the $r$-KdV hierarchy corresponding to the unique $r$-KdV solution whose initial condition is given explicitly in \cref{prop:ic} in terms of $\Theta^{r,s}$-integrals.
\end{introthm}

This provides a generalized version of the Brézin--Gross--Witten tau function of the KdV hierarchy \cite{BG80,GW80}, corresponding to the case $(r,s) = (2,1)$ and whose geometric interpretation in terms of intersection numbers was conjectured by Norbury \cite{Nor23} and proved in \cite{CGG25}. The $s=1$ case  of the above theorem gives an enumerative interpretation to the tau functions studied in \cite{YZ23} for certain choices of their constants $d_1,\ldots,d_{r-1}$ (see \cref{prop:YZ}). One can also view this statement as an analogue of the Witten $r$-spin conjecture (which corresponds to the case $s = r+1$) \cite{FSZ10}.

\subsection*{Loop equations and \texorpdfstring{$\mathcal{W}$}{W}-constraints}
We then move on to the study of $\mathcal{W}$-constraints. In~\cite{BBCCN24}, it is shown that the Bouchard--Eynard topological recursion on the $(r,s)$ spectral curves can be recast as a set of differential constraints satisfied by a partition function. These differential operators form a representation of the $\mathcal{W}(\mathfrak{gl}_r)$-algebra at self-dual level. The precise correspondence goes through the formulation of abstract loop equations satisfied by the correlators, to which topological recursion is the unique solution.

As the correlators that generate descendant integrals for the $\Theta^{r,s}$-classes satisfy the generalized topological recursion, but generally not the Bouchard--Eynard recursion, we cannot  apply the results of~\cite{BBCCN24} directly to obtain $\mathcal{W}$-constraints for the descendant potential. Instead, in \cref{sec:loop} we use the determinantal formulas that the generalized topological recursion correlators were proved to satisfy in \cite{ABDKS25-KP}, and follow the strategy of \cite{BEM18} to derive  loop equations for them explicitly.  These loop equations coincide with the usual ones when $s=r-1$, as expected from~\cite{CGG25}. When $s=1$, we obtain a particular case of the ``shifted loop equations'' studied recently in~\cite{BBKN25}. For other choices of $s$, we get a new system of loop equations.

We then recast those loop equations as differential constraints for the descendant potential.  These differential operators form a representation (the so-called twist field representation) of the principal $\mathcal{W}(\mathfrak{gl}_r)$-algebra at self-dual level. Thus, we obtain a set of $\mathcal{W}$-constraints for the descendant potential $Z^{r,s}$ in \cref{thm:W:const}.

\begin{introthm}[{$\mathcal{W}$-constraints}]
	Consider the modes $H^i_k$ of $\mathcal{W}(\mathfrak{gl}_r)$ in the representation \eqref{eq:Wmodes}. Then, for any $ r \geq 2 $ and $s \in [r-1]$, we have 
	\begin{equation}
		H^i_k Z^{r,s} = \begin{cases}
			\hbar^i A_i \delta_{k,0} Z^{r,s} & i \in [r-s] \, , k \geq 0, \\
			0 & r-s +1\leq i \leq r\, , k \geq r-s-i+1,
		\end{cases}
	\end{equation} 
	where the constants $A_i$ are defined in terms of the elementary symmetric polynomial $e_i$ by
	\begin{equation}
		A_i
		\coloneqq
		e_i\left(
		\frac{2+s-r-1}{2(r-s)}, \frac{4+ s-r-1}{2(r-s)}, \ldots, \frac{2(r-s) + s-r-1}{2(r-s)}
		\right),
	\end{equation}
	and vanish unless $i$ is even.
\end{introthm}

For $s=r-1$, the constraints match with the Airy structures studied in \cite{BBCCN24,CGG25}. For $s=1$, the constraints are a special case of the shifted Airy structures studied in \cite{BBKN25} and match the $\mathcal{W}$-constraints found in \cite{YZ23}. In both of these cases, they uniquely fix the descendant potential. However, for other choices of $s$, we obtain new $\mathcal{W}$-constraints which do not uniquely fix the descendant potential. However, all is not lost; we define a reduced descendant potential, which plays the role of ``initial conditions'' for the $\mathcal{W}$-constraints. Once the reduced potential is given, the constraints uniquely fix the entire descendant potential.

\subsection*{Notation}
For an integer $n \geq 1$, we use the notation $[n] = \set{ 1,2,\ldots,n }$, and the notation $z_{[n]}$ for the set $\set{ z_1,\ldots,z_n }$.

\section{Generalized topological recursion and determinantal formulas}
\label{sec:gen:TR:det:formulas}

In this section we give a geometric meaning for the correlators calculated by the
generalized topological recursion of~\cite{ABDKS25-gTR} on the $(r,s)$ spectral curve. We consider the nowhere semisimple cohomological field theory consisting of the classes $\set{ \Theta_{g,n}^{r,s} }_{g,n}$, which are defined as the top degrees of the Chiodo classes. We show that descendant integrals for the $\Theta^{r,s}$-classes are calculated by the generalized topological recursion on the $(r,s)$ spectral curve. This yields an explicit determinantal formula for these intersection numbers.

\subsection{Chiodo classes, topological recursion and determinantal formulas}
We first introduce the Chiodo classes and recall the main result of~\cite[Section~9.2.1]{Gia21} and~\cite[Proposition~3.1]{ABDKS24-logTR}, adjusted to our conventions and notation. Those results provide two different ways of computing intersection theory for the Chiodo classes, either through the Eynard--Orantin topological recursion or through a determinantal formula. 

\subsubsection{Chiodo classes}
Fix two integers $r, k$ and an integral $n$-tuple $a = (a_1,\dots,a_n) \in \Z^n$ satisfying $r \ge 2$ and the modular constraint $a_1 + \cdots + a_n \equiv k (2g-2+n) \pmod{r}$. Recall the definition of the moduli space $\smash{\Mbar_{g,a}^{r,k}}$ of twisted spin curves~\cite{Jar00,Chi08-twisted}: it parametrizes stable curves $(C,p_1,\dots,p_n,L)$ of genus $g$ with $n$ marked points and a line bundle $L$ on $C$ satisfying $L^{\otimes r} \cong \omega_{\log}^{\otimes k}(- \sum_i a_i p_i)$, where $\omega_{\log} \coloneqq \omega(\sum_{i} p_i)$ is the log-canonical bundle. It has a universal curve and a universal line bundle
\begin{equation}
	\pi \colon \mathcal{C} \longrightarrow \Mbar_{g,a}^{r,k},
	\qquad\qquad
	\mc{L} \longrightarrow \mathcal{C},
\end{equation}
and it comes with the forgetful map $\epsilon \colon \Mbar_{g,a}^{r,k} \to \Mbar_{g,n}$ to the moduli space of curves which drops the choice of the line bundle. The Chern polynomial of the derived pushforward of the universal line bundle gives a natural collection of classes on the moduli space of curves.

\begin{definition}
	Define the \emph{Chiodo class} as
	\begin{equation}
		C_{g,n}^{r,k+r}(a;\tau)
		\coloneqq
		\epsilon_{\ast} c(- \mathsf{R}^{\bullet}\pi_{\ast}\mc{L};\tau)
		\in R^{*}(\Mbar_{g,n})[\tau],
	\end{equation}
	where $\mathsf{R}^{\bullet}\pi_{\ast}\mc{L}$ is the derived pushforward of $\mc{L}$,
	\begin{equation}
		c(-\mathsf{E}^{\bullet};\tau)
		=
		\exp\left(
			\sum_{d \ge 1} (-\tau)^{d} (d-1)! \, \mathrm{ch}_d(\mathsf{E}^{\bullet})
		\right)
	\end{equation}
	is the Chern polynomial of $-\mathsf{E}^{\bullet} $ (where the minus sign should be interpreted as a minus sign in the Grothendieck ring, or alternatively a shift of the complex $\mathsf{E}^{\bullet} $ in the derived category), and $R^{*}(\Mbar_{g,n})$ is the tautological ring of the moduli space of curves. For the total Chern class (that is, when $\tau = 1$), we simply write $C_{g,n}^{r,k+r}(a)$.
\end{definition}

The fact that the Chiodo class belongs to the tautological ring is a consequence of Chiodo's formula~\cite{Chi08-class} for the Chern characters of $\mathsf{R}^{\bullet}\pi_{\ast}\mc{L}$, which in turn generalizes  Mumford's and Bini's computations~\cite{Mum83,Bin03}, for $r = 2$ and $k = -1$ and for $r = 2$ and arbitrary $k$ respectively, to arbitrary values of $r$. The analysis of the pushforward along the forgetful map $\epsilon$ was carried out in~\cite{JPPZ17}. See~\cite{GLN23} for a list of known properties satisfied by the Chiodo class. Of importance for us is that the Chiodo classes form a semisimple cohomological field theory (in general, without flat unit).

From now on, we fix $k = s-r$ with $s \in [r-1]$, which motivates the shift above. For $s$ in this range and $0 \leq a_i \leq r-1$, $i=1,\dots,n$, the restriction $L$ of the universal line bundle $\mc{L}$ to any geometric fiber has no global sections. Indeed, on a geometric fiber $(C,p_1,\ldots,p_n,L)$ we have
\begin{equation}
	L^{\otimes r}
	\cong
	\omega_{C,\log}^{\otimes(s-r)}
	\Biggl(-\sum_{i=1}^n a_i p_i\Biggr).
\end{equation}
Since $s-r<0$ and $a_i\geq 0$, the restriction of $L^{\otimes r}$ to every irreducible component of $C$ has negative degree. Hence $H^0(C,L^{\otimes r})=0$, and consequently $H^0(C,L)=0$ as well, since any nonzero section of $L$ would give, by taking its $r$-th power, a nonzero section of $L^{\otimes r}$. It follows that $\mathsf{R}^0\pi_*\mc{L}=0$, and therefore the complex $\mathsf{R}^{\bullet}\pi_*\mc{L}$ is  the vector bundle $\mathsf{R}^1\pi_*\mc{L}$ living in homological degree $1$. Riemann--Roch then gives its rank as
\begin{equation}\label{eq:deg}
	D^{r,s}_{g;a}
	\coloneqq
	g - 1 + \frac{(2g-2+n)(r-s)+|a|}{r}.
\end{equation}
Thus $C_{g,n}^{r,s}(a;\tau) = \sum_{d=0}^{D^{r,s}_{g;a}} [C_{g,n}^{r,s}(a)]_d \, \tau^{d}$, where by $[\alpha]_d$ we denote the $d$-th homogeneous component of a class $\alpha \in R^*(\overline{\mathcal{M}}_{g,n})$.

\subsubsection{Chiodo classes and topological recursion}
Since the Chiodo classes form a semisimple cohomological field theory, it follows from \cite{Eyn14,DOSS14} that descendant integrals of the Chiodo classes are calculated by the Eynard--Orantin topological recursion. 
The Eynard--Orantin topological recursion~\cite{EO07} takes as input a spectral curve, which consists of a Riemann surface $\Sigma$, two functions $x$ and $y$ on $\Sigma$,  together with a Bergman kernel $B$ on $\Sigma^2$. Furthermore, we assume that  $dx$ and $dy$ are meromorphic differentials whose zeros do not coincide, and that $dx$ only has simple zeros.
It outputs a collection of symmetric $n$-differentials $W_{g,n}(z_{[n]})$ on $\Sigma^n$ (called \emph{genus-$g$, $n$-point correlators}) for $g \in \mathbb{Z}_{\geq 0}$, $n \in \mathbb{Z}_{>0}$ and $2g-2+n>0$.

The relation between intersection theory of the Chiodo classes and topological recursion is stated in the following proposition, which was proved in \cite[Section~9.2.1]{Gia21}. It generalizes to arbitrary values of $t \in \C^*$ the computations of \cite{LPSZ17}, performed for $t = 1$, which in turn generalize the computations of \cite{SSZ15} performed in the case $t = 1$ and $s = r+1$.

\begin{proposition}[\cite{SSZ15,LPSZ17,Gia21}] \label{prop:Gia-t}
	Let $r\geq 2$, $s \in [r-1]$, and $0 \neq t \in \mathbb{C}$. Consider the spectral curve $(\Sigma,x,y,B)$, where
	\begin{equation}
		\Sigma = \mathbb{P}^1,
		\qquad
		x = z^r - t \log(z),
		\qquad
		y = z^{s-r},
		\qquad
		B = \frac{dz_1 dz_2}{(z_1-z_2)^2}.
	\end{equation}
	For $2g-2+n >0$, the correlators $W_{g,n}(z_{[n]};t)$ produced by the Eynard--Orantin topological recursion on this spectral curve take the form
	\begin{equation}
		W_{g,n}(z_{[n]};t)
		=
		\frac{1}{(s-r)^{2g-2+n}}
		\sum_{\substack{ k_1,\dots,k_n \ge 0 \\ a_1,\dots,a_n \in \set{0,\ldots, r-1} }}
			\int_{\overline{\mathcal{M}}_{g,n}}
				\left( \frac{t}{r} \right)^{D_{g;a}^{r,s}}
				C_{g,n}^{r,s}\left(a; \tfrac{r}{t}\right)
				\prod_{i=1}^n \psi_i^{k_i} d \theta_{k_i, a_i}(z_i;t),
	\end{equation}
	where the one-forms $d \theta_{k,a}$ are given by
	\begin{equation}\label{eq:theta}
		d \theta_{k,a}(z;t)
		\coloneqq
		r \, d \left( -\frac{d}{d x} \right)^{k+1} \frac{z^{r-a}}{r-a}
	\end{equation}
	for $x = z^r - t \log(z)$.
\end{proposition}

\begin{remark} 
	The restrictions $s \in [r-1]$ and $0 \leq a_i \leq r-1$ are not limiting. Indeed, shifting the parameter $s$ by $r$ results in a multiplication by $\kappa$-classes \cite{Eyn14}, while shifting $a_i$ by $r$ results in a multiplication by $\psi$-classes (see \cite[Theorem~4.1]{GLN23}). Thus, knowing the descendant integrals for parameters $s$ and $a_i$ in the above ranges suffices to recover all descendant integrals for Chiodo classes for arbitrary values of the parameters (except when $s \equiv 0 \pmod{r}$).
\end{remark}

We define the unstable correlators to be
\begin{equation}
	W_{0,1} \coloneqq -y\, dx = -z^{s-r} d(z^r-t \log(z)),
	\qquad\qquad
	W_{0,2} \coloneqq B = \frac{dz_1 dz_2}{(z_1-z_2)^2}.
\end{equation} The extra minus sign in the definition of the $(0,1)$ correlator is not standard, but it avoids extra minus signs appearing in various expressions in this paper. 
It will be useful to assemble the correlators into the generating series
\begin{equation}\label{eq:hbargs}
	W_n(z_{[n]};\hbar;t)
	\coloneqq
	\sum_{\substack{g \ge 0}} \hbar^{2g-2+n} W_{g,n}(z_{[n]};t),
	\qquad
	n \ge 1,
\end{equation}
called the \emph{$n$-point correlator}. We get:
\begin{multline} \label{eq:Gia-t}
	W_n(z_{[n]};\hbar;t) + \frac{1}{\hbar} \delta_{n,1}  z^{s-r}d(z^r-t\log z)-\delta_{n,2} \frac{dz_1dz_2}{(z_1-z_2)^2}
	= \\ 
	\sum_{\substack{g \ge 0\\ 2g-2+n>0}} \!\!
	\left( \frac{\hbar}{s-r} \right)^{2g-2+n} \!\!\!\!\!\!\!\!\!\!
	\sum_{\substack{ k_1,\dots,k_n \ge 0 \\ a_1,\dots,a_n \in \set{0,\ldots, r-1} }}
		\int_{\overline{\mathcal{M}}_{g,n}}
			\left( \frac{t}{r} \right)^{D^{r,s}_{g;a}} \,
			C_{g,n}^{r,s}(a;\tfrac{r}{t})
		\prod_{i=1}^n \psi_i^{k_i} \, d \theta_{k_i,a_i}(z_i;t). 
\end{multline} 

\subsubsection{Chiodo classes, generalized topological recursion and determinantal formulas}
It was found more recently in~\cite{ABDKS24-logTR} that the same differentials can be computed through a determinantal formula. The construction is a special case of the framework of generalized topological recursion developed in~\cite{ABDKS25-gTR}.

Just like the Eynard--Orantin topological recursion, the generalized topological recursion takes as input a spectral curve. In this context, it consists of a Riemann surface $\Sigma$, two differentials $dx$ and $dy$ on $\Sigma$, a Bergman kernel $B$ on $\Sigma^2$, and a set $\mathcal{P} \subset \Sigma$ which is a subset of the so-called set of special points. The special points are those $o \in \Sigma$ where the local behavior of $dx$ and $dy$ is given by
\begin{align}
	dx(z) = a z^{p-1} \, dz \, \bigl( 1 + \bigO(z) \bigr),
	\qquad
	dy(z) = b z^{q-1} \, dz \, \bigl( 1 + \bigO(z) \bigr),
	\qquad
	a,b \ne 0, \
	p, q \in \Z,
\end{align}
where $z$ is a local coordinate near $o$, with $p+q>0$ and $p$ and $q$ not simultaneously equal to $1$. 

For brevity, we bypass the definition of the generalized topological recursion here and simply state the resulting determinantal formula for the differentials in the case of interest. Let $r\geq 2$, $s \in [r-1]$, and $ t \in \mathbb{C}^*$. Consider the spectral curve $(\Sigma,dx,dy,B)$, where
\begin{equation}
	\Sigma = \mathbb{P}^1,
	\;
	dx = d(z^r - t \log(z)),
	\;
	dy = d(z^{s-r}),
	\;
	B = \frac{dz_1 dz_2}{(z_1-z_2)^2},
	\;
	\mathcal{P} = \Set{ z = e^{\frac{2 \pi a \iu}{r}} \left(\tfrac{t}{r}\right)^{1/r} }_{a \in [r]}.
\end{equation}
Here a choice of $r$-th root of $t/r$ is fixed once and for all. The special points of the above spectral curve are at $z = e^{2 \pi a \iu/r} ( t/r )^{1/r}$, $a \in [r]$. Thus, we choose $\mathcal{P}$ to be the full set of special points. The generalized topological recursion then constructs differentials $\omega_{n}(z_{[n]};\hbar;t)$ through the following determinantal formula (see~\cite[Theorem~3.6]{ABDKS25-gTR}):
\begin{equation} \label{eq:CloseFormula-t}
	\omega_n(z_{[n]};\hbar;t)
	\coloneqq
	\left( \prod_{i=1}^n \mathsf{O}_{x_i} \right)
	\sum_{\sigma\in C_n}
		\sgn(\sigma)
		\prod_{i=1}^n
		\frac{\sqrt{dw^+_i dw^-_{\sigma(i)}}}{w^+_i-w^-_{\sigma(i)}},
\end{equation}
where $\mathsf{O}_{x_i}$ and $w^\pm_i$ are the operators and functions given by
\begin{equation}\label{eq:O}
	\mathsf{O}_{x_i}
	\coloneqq
	-
	\sum_{k=-1}^\infty
		\left(d_i \frac{1}{dx_i} \right)^k \,
		[u_i^k] \, e^{ -u_i(\mc{S}(u_i\hbar\partial_{y_i})x_i-x_i) }
	\qquad\text{and}\qquad
	w^\pm_i
	\coloneqq
	e^{\pm \frac{1}{2} u_i \hbar \partial_{y_i}} z_i,
\end{equation}
where $d_i$ denotes the exterior derivative with respect to the variable $z_i$, $\mc{S}(u) = u^{-1}(e^{u/2} - e^{-u/2})$, $x_i = z_i^r-t\log z_i$, $y_i = z_i^{s-r}$, and $C_n \subset S_n$ denotes the subset of permutations consisting of $n$-cycles.

We remark that the $k = -1$ term in the summation in the definition of $\mathsf{O}_x$ in~\cref{eq:O} may appear unusual, but it is only used in the case $n = 1$, where it serves to convert
\begin{equation}
	\frac{\sqrt{dw^+_1 dw^-_1}}{w^+_1 - w^-_1} = \frac{dy_1}{u_1 \hbar}
	\qquad\text{into}\qquad
	-\left(d_1 \frac{1}{dx_1}\right)^{-1} dy_1 = -y_1 dx_1.
\end{equation}
An alternative presentation would be to define $\mathsf{O}_{x_i}$ using the sum $\sum_{k = 0}^\infty$ and then include an explicit term $- \hbar^{-1} \delta_{n,1} y_1 dx_1$ on the right-hand side of~\eqref{eq:CloseFormula-t}.

The main result is that, for $t \in \mathbb C^*$, the $n$-point correlators $W_n(z_{[n]};\hbar;t)$ constructed from the Eynard--Orantin topological recursion and the $n$-point correlators $\omega_n(z_{[n]}; \hbar; t)$ constructed by generalized topological recursion coincide.

\begin{proposition}[\cite{ABDKS24-logTR}] \label{prop:ClosedFormula-t}
	For $t \in \mathbb C^*$,
	\begin{equation}
		W_n(z_{[n]}; \hbar; t)= \omega_n(z_{[n]};\hbar;t),
	\end{equation}
	where $W_n(z_{[n]}; \hbar; t)$ is given by \eqref{eq:Gia-t} and $\omega_n(z_{[n]};\hbar;t)$ is given by \eqref{eq:CloseFormula-t}. In other words, the descendant integrals of the Chiodo classes are calculated by the determinantal formula \eqref{eq:CloseFormula-t}.
\end{proposition}

\begin{proof}
	This follows from a key property of generalized topological recursion: when the set $\mathcal{P}$ consists solely of simple zeros of $dx$, and $dy$ does not vanish on $\mathcal{P}$, the differentials produced by the generalized topological recursion coincide with those of the original topological recursion of~\cite{EO07}; see~\cite[Theorem 2.15]{ABDKS25-gTR}. Since this condition holds for the spectral curve considered here (as $t \neq 0$), the proposition follows.
\end{proof}

\subsection{\texorpdfstring{$\Theta^{r,s}$}{Theta}-classes, generalized topological recursion and determinantal formulas}
\Cref{prop:ClosedFormula-t} says that intersection theory for the Chiodo classes (for $t \neq 0$) is calculated by the determinantal formula \eqref{eq:CloseFormula-t}. In this section we study the limit as $t \to 0$. We show that the intersection theory of the top degrees of the Chiodo classes is also given by a determinantal formula and obtained via the generalized topological recursion.

\subsubsection{\texorpdfstring{$\Theta^{r,s}$}{Theta}-classes}
The $t \to 0$ limit of \cref{prop:Gia-t} motivates the definition of the $\Theta^{r,s}$-classes, which are the top degrees of the Chiodo classes.

\begin{definition}\label{def:theta}
	Define the classes
	\begin{equation}
		\Theta^{r,s}_{g,n}(a)
		\coloneqq
		\frac{r^{\frac{(2g-2+n)(r-s)+|a|}{r}}}{(s-r)^{2g-2+n}}
		\bigl[ C_{g,n}^{r,s}(a) \bigr]_{D_{g;a}^{r,s}} \,,
	\end{equation}
	which are the (appropriately normalized) top degrees of the Chiodo classes.
\end{definition} 

As outlined in~\cite{CGG25}, the collection of classes $\set{ \Theta^{r,s}_{g,n}(a) }_{g,n}$ with $a \in [r-1]^n$ defines a cohomological field theory of rank $r - 1$. More precisely, consider the vector space $ V \coloneqq \operatorname{span}_{\mathbb Q} (v_1,\ldots,v_{r-1})$, along with the pairing 
\begin{equation}
	\eta (v_{a_1},v_{a_2}) = \delta_{a_1+a_2,r}.
\end{equation}

\begin{proposition}\label{prop:CohFT}
	The collection of maps $\Theta^{r,s}_{g,n} \colon V^{\otimes n} \to R^* (\overline{\mathcal{M}}_{g,n}) $ given by the assignment
	\begin{equation}
		v_{a_1} \otimes \cdots \otimes v_{a_n} \longmapsto \Theta^{r,s}_{g,n}(a_1,\ldots,a_n)
	\end{equation}
	and extended by linearity define a cohomological field theory of rank $r - 1$ without a flat unit on $(V,\eta)$. Moreover, this cohomological field theory satisfies the following modified unit axiom:
	\begin{equation}\label{eq:modunit}
		\Theta^{r,s}_{g,n+1}(a_1,\ldots,a_n,s)
		=
		\psi_{n+1} \cdot p^* \Theta^{r,s}_{g,n} (a_1,\ldots,a_n),
	\end{equation}
	where $p \colon \overline{\mathcal{M}}_{g,n+1} \to \overline{\mathcal{M}}_{g,n}$ is the forgetful map that forgets the last marked point.
\end{proposition}

\begin{proof}
	The proof of \cite[Proposition~2.6 and Theorem~2.7]{CGG25} for the case $s = r-1$ can be followed without any change for any $s \in [r-1]$.
\end{proof}

Note the exclusion of insertions (also known as primary fields) with $a_i = 0$, which reduces the rank by one. Dropping these insertions is essential to obtain a cohomological field theory.

\begin{remark}
	Assuming that the $a_i$ are in the range $ 1 \leq a_i \leq r-1$, $\Theta^{r,s}_{g,n}(a)$ when $s  \leq 0 $ is of degree bigger than $ 3g-3+n $ and hence vanishes identically. On the other hand, when $s \geq r$, $\mathsf{R}^\bullet \pi_* \mathcal L$ is generically a complex and not  a vector bundle and thus the Chiodo class has terms in all cohomological degrees. This explains why we restrict ourselves to the range $s \in [r-1]$.
\end{remark}

We also note that the Dubrovin--Frobenius manifold defined by the cohomological field theory above is nowhere semisimple (for instance, $v_{r-1}$ is always a nilpotent element), and thus Teleman's classification result~\cite{Tel12} does not apply. A workaround was proposed in~\cite{CGG25} for the case $s = r - 1$, using a deformation argument. This not only provides a method for computing the descendant integrals, but also produces relations in cohomology. 

Here we propose a different method for computing the intersection numbers based on the $t \to 0$ limit of \cref{prop:ClosedFormula-t}. Although this method does not address the class itself, only the associated intersection numbers, it is much more direct. 

\subsubsection{The \texorpdfstring{$t \to 0$}{t to zero} limit}
First we recall the definition of the multi-factorial. For $m,r \in \mathbb{Z}_{>0}$,
\begin{equation}
	m!^{(r)}
	\coloneqq
	\begin{cases}
		m (m-r)!^{(r)} & m > r, \\
		m & 1 \leq m \leq r.
	\end{cases}
\end{equation}

\begin{theorem}[Determinantal formulas]\label{thm:WgnTheta}
	Let $W_{n}(z_{[n]};\hbar;t)$ be given by \eqref{eq:Gia-t} and $\omega_n(z_{[n]};\hbar;t)$ be given by \eqref{eq:CloseFormula-t}. Then
	\begin{equation}\label{eq:limit}
		\lim_{t \to 0} \; W_{n}(z_{[n]};\hbar;t)
		=
		\lim_{t \to 0} \; \omega_n(z_{[n]};\hbar;t).
	\end{equation}
	As a result, the descendant integrals of the $\Theta^{r,s}$-classes can be computed by the following determinantal formula:
	\begin{multline}\label{eq:Theta:rs}
		\sum_{\substack{g \ge 0 \\ 2g-2+n>0}} \left(\frac{\hbar}{r}\right)^{2g-2+n}
		\sum_{\substack{ k_1,\dots,k_n \ge 0 \\ a_1,\dots,a_n \in [r-1] }}
			\int_{\overline{\mathcal{M}}_{g,n}}
				\Theta_{g,n}^{r,s}(a)
			\prod_{i=1}^n \psi_i^{k_i} \, d \xi_{k_i,a_i}(z_i) = \\
		\left( \prod_{i=1}^n \mathsf{O}_{x_i} \right)
		\sum_{\sigma\in C_n}
			\sgn(\sigma)
			\prod_{i=1}^n
			\frac{\sqrt{dw^+_i dw^-_{\sigma(i)}}}{w^+_i-w^-_{\sigma(i)}} + \frac{1}{\hbar} \delta_{n,1} z_1^{s-r}d (z_1^r) - \delta_{n,2} \frac{dz_1dz_2}{(z_1-z_2)^2}, 
	\end{multline}
	where $\mathsf{O}_{x_i}$, and $w^\pm_i$ are given by \cref{eq:O} for the choices $x_i = z_i^{r}$ and $y_i = z_i^{s-r}$, and the one-forms $d \xi_{k,a}(z)$ are given by
	\begin{equation}
		d \xi_{k,a}(z)
		\coloneqq
		(rk+a)!^{(r)} \frac{dz}{z^{rk+a+1}}.
	\end{equation}
\end{theorem}

\begin{proof}
	First, both \cref{eq:Gia-t,eq:CloseFormula-t} are rational functions in $z_1,\ldots,z_n \in \P^1$ and $t \in \C^*$, and for both formulas $t = 0$ is a removable singularity. Since the two formulas are equal by \cref{prop:ClosedFormula-t}, it implies that their limits for $t \to 0$ are equal as well, which is \eqref{eq:limit}.

	To obtain the rest of the theorem, we evaluate the limits explicitly. First, we evaluate the limit of $W_{n}(z_{[n]};\hbar;t)$ using \eqref{eq:Gia-t}. We have:
	\begin{equation}
	\begin{split}
		\frac{1}{(s-r)^{2g-2+n}}
		\lim_{t\to 0}
			\left( \frac{t}{r} \right)^{D^{r,s}_{g;a}} \,
			C_{g,n}^{r,s}(a;\tfrac{r}{t})
		&=
		\frac{1}{(s-r)^{2g-2+n}}\bigl[ C_{g,n}^{r,s}(a) \bigr]_{D_{g;a}^{r,s}} \\
		&=
		r^{-\frac{(2g-2+n)(r-s)+|a|}{r}} \Theta_{g,n}^{r,s}(a) .
	\end{split}
	\end{equation}
	Moreover, for the primary field $a=0$ we have $d \theta_{k,0}(z; t) \to d \left(- \frac{d}{dx}\right)^{k+1} z^r$ as $t \to 0$. But
	\begin{equation}
		\frac d{dx} z^r = \frac{r z^r}{rz^r-t} = 1 + \bigO(t), 
	\end{equation} 
	and hence $d \theta_{k,0}(z; t) \to 0$ for all $k \geq 0$. Therefore, we omit the primary field $a_i=0$ in the arguments of $\Theta^{r,s}_{g,n}(a)$. For $a \in [r-1]$ we have 
	\begin{equation}
	\begin{split}
		\lim_{t \to 0} \; d \theta_{k,a}(z;t) 
		&=
		\lim_{t\to 0} \; r\,d \Big(-\frac{z}{rz^r-t} \frac{d}{dz}\Big)^{k+1} \frac{z^{r-a}}{r-a} \\
		&=
		r\,d \Big(-\frac{1}{rz^{r-1}} \frac{d}{dz}\Big)^{k+1} \frac{z^{r-a}}{r-a} \\ 
		&=
		\frac{(rk+a)!^{(r)}}{r^k} \frac{dz}{z^{rk+a+1}}. 	
	\end{split}
	\end{equation}
	Defining $d \xi_{k,a}(z) = (rk+a)!^{(r)} \frac{dz}{z^{rk+a+1}}$ and putting all of this together, we get that
	\begin{multline}
		\lim_{t \to 0} \;
			W_{n}(z_{[n]};\hbar;t)
		+ \frac{1}{\hbar}\delta_{n,1} z^{s-r}d(z^r)
		- \delta_{n,2} \frac{dz_1dz_2}{(z_1-z_2)^2}
		= \\
		\sum_{g \ge 0}  \hbar^{2g-2+n}
			\sum_{\substack{ k_1,\dots,k_n \ge 0 \\ a_1,\dots,a_n \in [r-1] }} \frac{1}{r^{\sum_{i=1}^n k_i + \frac{(2g-2+n)(r-s)+|a|}{r}}}
				\int_{\overline{\mathcal{M}}_{g,n}}
					\Theta_{g,n}^{r,s}(a)
				\prod_{i=1}^n \psi_i^{k_i} \, d \xi_{k_i,a_i}(z_i).
	\end{multline}
	We finally notice that, for degree reasons, the integrals are necessarily zero unless 
	\begin{equation}
		3g-3+n = D_{g;a}^{r,s} + \sum_{i=1}^n k_i = g-1 + \frac{(2g-2+n)(r-s)+|a|}{r} + \sum_{i=1}^n k_i,
	\end{equation}
	and so we can replace the exponent of $r$ by $2g-2+n$ to get
	\begin{multline}
		\lim_{t \to 0} \; W_{n}(z_{[n]};\hbar;t)
		+ \frac{1}{\hbar} \delta_{n,1} z^{s-r}d(z^r)
		- \delta_{n,2} \frac{dz_1dz_2}{(z_1-z_2)^2}
		= \\
		\sum_{g \ge 0} \left( \frac{\hbar}{r} \right)^{2g-2+n}
			\sum_{\substack{ k_1,\dots,k_n \ge 0 \\ a_1,\dots,a_n \in [r-1]}}
				\int_{\overline{\mathcal{M}}_{g,n}}
					\Theta_{g,n}^{r,s}(a)
				\prod_{i=1}^n \psi_i^{k_i} \, d \xi_{k_i,a_i}(z_i).
	\end{multline}
	As for the limit of $\omega_n(z_{[n]};\hbar;t)$, it is straightforward, as the dependence on $t$ in~\eqref{eq:CloseFormula-t} is hidden in the definition of $\mathsf{O}_{x}$, and its limit for $t \to 0$ is manifestly $\mathsf{O}_{x}$ used in~\eqref{eq:Theta:rs}.
\end{proof}

\subsubsection{\texorpdfstring{$\Theta^{r,s}$}{Theta}-classes and generalized topological recursion}
\Cref{thm:WgnTheta} gives a determinantal formula for descendant integrals of the $\Theta^{r,s}$-classes. It was obtained by taking the $t \to 0$ limit of \Cref{prop:ClosedFormula-t}. For $t \neq 0$, the differentials $W_{n}(z_{[n]};\hbar;t)$ from \eqref{eq:Gia-t} were obtained via the Eynard--Orantin topological recursion, while the differentials $\omega_n(z_{[n]}; \hbar; t)$ of \eqref{eq:CloseFormula-t} were obtained via the generalized topological recursion. It is natural to ask whether similar statements can be made for the limit $t \to 0$: are the differentials $\lim_{t \to 0} W_{n}(z_{[n]};\hbar;t)$ and $\lim_{t\to 0} \omega_n(z_{[n]}; \hbar; t)$ computed by topological recursion?

This question is easy to answer from the generalized topological recursion side.

\begin{definition}\label{def:rs}
	Let $r \geq 2$ and $s \in [r-1]$. We define the \emph{$(r,s)$ spectral curve} by the data
	\begin{equation}
		\Sigma = \mathbb{P}^1,
		\quad
		dx = d(z^r ),
		\quad
		dy = d(z^{s-r}),
		\quad
		B = \frac{dz_1 dz_2}{(z_1-z_2)^2},
		\quad
		\mathcal{P} = \set{ z=0 }.
	\end{equation}
\end{definition}

This is the $t \to 0$ limit of the spectral curve considered previously.

\begin{corollary}[Generalized topological recursion]\label{thm:gTR}
	The $n$-point correlators $\omega_n(z_{[n]};\hbar)$ produced by generalized topological recursion on the $(r,s)$ spectral curve calculate the descendant integrals of the $\Theta^{r,s}$-classes:
	\begin{multline}\label{eq:Theta:rs:omega}
		\omega_n(z_{[n]};\hbar)
		=
		-\frac{1}{\hbar} \delta_{n,1} z_1^{s-r}dz_1^r + \delta_{n,2} \frac{dz_1dz_2}{(z_1-z_2)^2} +
		\\
		+
		\sum_{\substack{g \ge 0 \\ 2g-2+n>0}} \left( \frac{\hbar}{r} \right)^{2g-2+n}
		\sum_{\substack{ k_1,\dots,k_n \ge 0 \\ a_1,\dots,a_n \in [r-1]}}
			\int_{\overline{\mathcal{M}}_{g,n}}
				\Theta_{g,n}^{r,s}(a)
				\prod_{i=1}^n \psi_i^{k_i} \, d \xi_{k_i,a_i}(z_i). 
	\end{multline}
\end{corollary}

\begin{proof}
	Looking back at \cref{prop:ClosedFormula-t}, what the present corollary is essentially saying is that the limit $\lim_{t \to 0} \omega_n(z_{[n]};\hbar;t)$ of the differentials constructed by the generalized topological recursion on the spectral curve with $dx = d (z^r-t \log(z))$ and $dy = d(z^{s-r})$ are equal to the differentials constructed by the generalized topological recursion on the limiting $t\to 0$ spectral curve with $dx = d(z^r)$ and $dy = d(z^{s-r})$. But this follows from the following key property of generalized topological recursion. 

	If the input data of generalized topological recursion varies analytically with respect to a parameter $t$ such that all points in $\mathcal{P}_t$ belong to a domain $U$ such that $dx$ and $dy$ are regular and non-vanishing along the boundary of $U$ and there are no further special points in $U$, then the differentials of generalized topological recursion depend on $t$ analytically, see~\cite[Theorem~5.3]{ABDKS25-gTR}. This implies that the $t \to 0$ limit of the differentials produced by generalized topological recursion coincide with the differentials on the limiting spectral curve, which proves the statement.
\end{proof}

What about the limit $\lim_{t \to 0} W_{n}(z_{[n]}; \hbar; t)$? Are these correlators computed by topological recursion on the limiting spectral curve? It turns out that this question is quite subtle. First, in the $t \to 0$ limit, the spectral curve becomes $x = z^r$, $y = z^{s - r}$. Then, $dx$ has a single zero at $z = 0$, which, for $r > 2$, is not simple. Hence, the original Eynard--Orantin topological recursion does not apply. One can instead consider the Bouchard--Eynard topological recursion~\cite{BE13}. Do the correlators $\lim_{t \to 0} W_{g,n}(z_{[n]}; t)$ then coincide with those produced by the Bouchard--Eynard recursion on the limiting spectral curve? This type of question was studied in detail in~\cite{BBCKS25}. In this particular case, the answer is generally no: the two coincide only when $s = r - 1$, which is precisely the case studied in~\cite{CGG25}. In other words, while generalized topological recursion naturally ``commutes with limits'' (with a careful choice of $\mathcal P$), the Eynard--Orantin recursion---and its Bouchard--Eynard generalization---does not always do so.

Therefore, although generalized topological recursion on the $(r, s)$ spectral curve computes the descendant integrals of the $\Theta^{r,s}$-classes, the Bouchard--Eynard topological recursion generally computes something else, except in the case $s = r - 1$, where the two coincide. A geometric interpretation of the Bouchard--Eynard topological recursion for other values of $s$ remains unknown.

\section{Integrability}
\label{sec:integrability}

In the previous section, we showed that generalized topological recursion on the $(r,s)$ spectral curve produces differentials $\omega_n(z_{[n]}; \hbar)$ that compute the descendant integrals of the $\Theta^{r,s}$-classes. This leads to a clear integrability statement, which we explore in this section.

\subsection{\texorpdfstring{$r$}{r}-KdV integrability}
A universal property of generalized topological recursion on genus zero spectral curves is that it always produces systems of differentials that are KP integrable, see~\cite[Theorem~1.1]{ABDKS26-KPg0} and~\cite[Theorem~6.4]{ABDKS25-gTR}. 

In our case of interest, the resulting statement is that the partition function associated to the expansion of the differentials $\omega_n(z_{[n]};\hbar)$ on the $(r,s)$ spectral curve at any regular point is a KP tau function. In particular, the expansion of the differentials $\omega_n(z_{[n]};\hbar)$ at $z = \infty$ in the local coordinate \smash{$z^{-1}$} gives, up to a re-scaling of the times, the following tau function:
\begin{equation}\label{eq:tau}
	Z^{r,s}(\bm{x};\hbar)
	\coloneqq
	\exp \left(
	\sum_{\substack{g \geq 0, \, n \geq 1 \\ 2g-2+n>0}} \frac{\hbar^{2g-2+n}}{n!}
	\sum_{\substack{k_1,\ldots,k_n \geq 0 \\ a_1,\ldots,a_n \in [r-1]}}
	\int_{\Mbar_{g,n}}
	\Theta^{r,s}_{g,n}(a)
	\prod_{i=1}^n \psi_i^{k_i}  (rk_i+a_i)!^{(r)} x_{rk_i + a_i}
	\right). 
\end{equation}

\begin{remark}
	While the appearance of the multi-factorials may look strange at first, they appear naturally if we allow the sum over $a_i $ to be over all $a_i > 0$ and omit $\psi$-classes. Indeed we can rewrite $ Z^{r,s}$ as 
	\begin{equation}
		Z^{r,s}(\bm{x};\hbar)
		=
		\exp \left(
		\sum_{\substack{g \geq 0, \, n \geq 1 \\ 2g-2+n>0}} \frac{\hbar^{2g-2+n}}{n!}
		\sum_{ a_1,\ldots,a_n \geq 1 }
		\int_{\Mbar_{g,n}}
		\Theta^{r,s}_{g,n}(a)
		\prod_{i=1}^n   \langle a_i\rangle  x_{ a_i}
		\right),
	\end{equation}
	where $\langle a\rangle$ denotes the remainder of the Euclidean division of $a$ by $r$. This formula follows from the properties of the Chiodo class under the shift of $a $ by $r$ as proved in \cite[Theorem~4.1(ii)]{GLN23}.
\end{remark}

We thus obtain the following integrability statement.

\begin{theorem}[Integrability] \label{thm:int}
	The descendant potential 	$Z^{r,s}(\bm{x};\hbar)$ of the $\Theta^{r,s}$-classes is a tau function of the $r$-KdV (also known as $r$-Gelfand--Dickey) hierarchy.
\end{theorem}

\begin{proof}
	This is an immediate observation combining the KP integrability of the partition function in~\eqref{eq:tau} with its independence from the times $x_{rm}$, for $m \geq 1$. The latter follows from the absence of monomials $z^{rm}$ in $d\xi_{k,a}(z)$ in~\eqref{eq:Theta:rs:omega}, since  $a$ only takes values in $[r - 1]$, i.e., $a = 0$ is excluded.
\end{proof}

\subsection{The initial conditions}
\label{ssec:ic}
\Cref{thm:int} says that $Z^{r,s}(\bm{x};\hbar)$ is an $r$-KdV tau function, but it does not say \emph{which} tau function. To specify the tau function we need to compute initial conditions for the normal coordinates of the $r$-KdV hierarchy. In other words, we want to compute 
\begin{equation}\label{eq:ualpha:def}
	u_\alpha(x_1;\hbar)
	\coloneqq
	\partial_{1}\partial_{\alpha} \log Z^{r,s}(\bm{x};\hbar) \big|_{
		x_{i} = 0 \, \forall \, i \geq 2
	}
\end{equation}
for all $\alpha \in [r-1]$. Here $\partial_{m}$ is the shorthand notation for \smash{$\frac{\partial}{\partial x_m}$}. The functions $\partial_{1}\partial_{\alpha} \log Z^{r,s}$ are called the normal coordinates of the $r$-KdV hierarchy, so $u_\alpha$ are the corresponding initial conditions. From the definition of $Z^{r,s}$, it follows that
\begin{equation}\label{eq:ualpha}
	u_\alpha(x_1;\hbar)
	=
	\sum_{\substack{g \geq 0, \, n \geq 2 \\ 2g-2+n>0}}
	\frac{\hbar^{2g-2+n}}{(n-2)!}
	\int_{\Mbar_{g,n}}
	\Theta^{r,s}_{g,n}(\alpha, 1,\ldots,1) \, \alpha \, x_1^{n-2} .
\end{equation}
We can express the initial conditions for the normal coordinates in terms of finitely many primary integrals (i.e. no $\psi$-classes) of the $\Theta^{r,s}$-class. The shape of these initial conditions looks different for $s = 1$ and $s > 1$.

\begin{proposition}\label{prop:ic}
	The following initial conditions hold.
	\begin{itemize}
		\item For any $r \ge 2$, $s = 1$, and $\alpha \in [r-1]$:
		\begin{equation}
			u_\alpha(x_1;\hbar)
			=
			\delta_{\alpha,2g-1}
			\left( \alpha^2 \int_{\Mbar_{g,1}} \Theta^{r,1}_{g,1}(\alpha) \right)
			\left( \frac{\hbar}{1 - \hbar x_1} \right)^{\alpha+1} .
		\end{equation}
		
		\item For any $r \ge 3$, $s =2,\ldots,r-1$, and $\alpha \in [r-1]$:
		\begin{equation}
			u_\alpha(x_1;\hbar)
			=
			\sum_{\substack{g \geq 0, \, n \geq 2 \\ 2g-2+n>0 \\ \alpha = s(2g-2+n)-n+1}}
			\frac{\hbar^{2g-2+n}}{(n-2)!}
			\int_{\Mbar_{g,n}}
			\Theta^{r,s}_{g,n}(\alpha, 1,\ldots,1) \, \alpha \, x_1^{n-2} ,
		\end{equation}
		which is a polynomial in $x_1$ and $\hbar$.
	\end{itemize}
\end{proposition}

\begin{proof}
	We want to understand when the integral in \eqref{eq:ualpha} is non-zero. For degree reasons, the integral is automatically zero unless
	\begin{equation}\label{eq:degree}
		3g - 3+ n = g-1+\frac{(2g-2+n)(r-s)+\alpha + n-1}{r},
	\end{equation}
	which simplifies to $\alpha = s(2g-2+n)-n+1$. For $s > 1$, there are only finitely many $(g,n)$ that satisfy the above equation. Thus, we see  from \cref{eq:ualpha} that for $s > 1$, $u_\alpha$
	is a polynomial in $x_1$ and $\hbar$. This proves the second point of the statement.

	For $s = 1$, the condition \eqref{eq:degree} simplifies further to $\alpha = 2g-1$, which is independent of $n$. As a result, there are infinitely many values of $n$ to take into account. Recall from \cref{prop:CohFT} that the $\Theta^{r,s}$-classes satisfy the modified unit axiom, that is
	\begin{equation}
		\Theta^{r,1}_{g,n+1}(a_1,\ldots, a_n,1)
		= 
		\psi_{n+1} \cdot p^* \Theta^{r,1}_{g,n}(a_1,\ldots, a_n) ,
	\end{equation}
	where $p \colon \Mbar_{g,n+1} \to \Mbar_{g,n} $ is the forgetful map that forgets the last marked point. Integrating over $\Mbar_{g,n+1}$, we get
	\begin{equation}
		\int_ {\Mbar_{g,n+1}} \Theta^{r,1}_{g,n+1}(a_1,\ldots, a_n,1)
		=
		(2g-2+n) \int_ {\Mbar_{g,n}} \Theta^{r,1}_{g,n}(a_1,\ldots, a_n)
	\end{equation}
	by applying the projection formula. As a consequence, we can compute $u_\alpha$ for the case $s=1$ as 
	\begin{equation}
		\begin{split}
			u_\alpha(x_1;\hbar)
			&=
			\delta_{\alpha,2g-1}
			\sum_{n \geq 1} \frac{\hbar^{2g-1+n}}{(n-1)!}
			\int_{\Mbar_{g,n+1}} \Theta^{r,1}_{g,n+1}(\alpha, 1,\ldots,1) \alpha x_1^{n-1} \\
			&=
			\delta_{\alpha,2g-1}
			\left(\alpha  \int_{\Mbar_{g,1}} \Theta^{r,1}_{g,1}(\alpha) \right)
			\sum_{n \geq 1} \hbar^{\alpha+1} \frac{(2g-2+n)\cdots (2g-2+1)}{(n-1)!} (\hbar x_1)^{n-1} \\
			&=
			\delta_{\alpha,2g-1}
			\left( \alpha \int_{\Mbar_{g,1}} \Theta^{r,1}_{g,1}(\alpha) \right)
			\sum_{n \geq 1} \hbar^{\alpha+1} \frac{(\alpha+n-1)!}{(\alpha-1)!(n-1)!} (\hbar x_1)^{n-1} \\
			&=
			\delta_{\alpha,2g-1}
			\left( \alpha^2 \int_{\Mbar_{g,1}} \Theta^{r,1}_{g,1}(\alpha) \right)
			\left( \frac{\hbar}{1 - \hbar x_1} \right)^{\alpha+1} .
		\end{split}
	\end{equation}
	This proves the first point of the statement.
\end{proof}

\begin{remark}\label{rem:YZ}
	It is useful to compare the initial condition for $s = 1$ with that of the generalized Brézin--Gross--Witten model studied in~\cite{YZ23} from the point of view of integrable systems. In \textit{loc. cit.}, the authors consider the solution to the $r$-KdV hierarchy $u_{\alpha}$ subject to the initial condition\footnote{
		In~\cite{YZ23}, there is no explicit $\hbar$-dependence. This can be restored by a homogeneity analysis.
	}
	\begin{equation}
		u_{\alpha}(x_1;\hbar) = d_{\alpha} \left( \frac{\hbar}{1 - \hbar x_1} \right)^{\alpha+1},
		\qquad
		\alpha \in [r-1],
	\end{equation}
	for arbitrary constants $d_{\alpha} \in \mathbb{C}$. Clearly, the $r$-KdV solution $\partial_{1}\partial_{\alpha} \log Z^{r,1}$, which computes the descendant integrals of the $\Theta^{r,1}$-classes, is a special case of the solution studied by Yang--Zhou, corresponding to the choice of constants \smash{$d_{\alpha} = \delta_{\alpha,2g-1} \, \alpha^2 \int_{\Mbar_{g,1}} \Theta^{r,1}_{g,1}(\alpha)$}. We will return to a discussion of the tau function in \cref{sec:W:constraints}.
\end{remark}

\section{Loop equations}
\label{sec:loop}
In this section, we study further the generalized topological recursion correlators $\omega_{n}(z_{[n]};\hbar)$, which encode descendant integrals of the $\Theta^{r,s}$-classes. We derive loop equations for the $\omega_{g,n}$ that resemble the loop equations satisfied by the Bouchard--Eynard topological recursion correlators on the $(r,s)$ spectral curve, but are different except in the special case of $s = r-1$. However, when $s=1$, the loop equations that we obtain are particular cases of the shifted loop equations studied in~\cite{BBKN25}, and consequently, the correlators $\omega_{g,n}$ for $s=1$ can be computed by the shifted topological recursion defined in~\cite{BBKN25}.

\subsection{A second determinantal formula}
As we have seen, the differentials $\omega_n(z_{[n]};\hbar)$ produced by generalized topological recursion on the $(r,s)$-curve satisfy the determinantal formula~\eqref{eq:Theta:rs}. However, a consequence of~\cite[Section~3]{ABDKS25-KP} is an alternative expression for the $n$-point function as a determinantal formula in terms of the so-called wave functions.

\subsubsection{Wave functions and Baker--Akhiezer kernel}
Define the formal functions
\begin{equation}
	\phi_k(z;\hbar)
	\coloneqq
	\sqrt{x'(z)} \int \zeta^{k-1} \sqrt{y'(\zeta)} e^{-\frac{1}{\hbar} \bigl( x(z)(y(z)-y(\zeta)) + \int_{z}^{\zeta} x dy \bigr)} \frac{d\zeta}{\sqrt{2\pi\hbar}},
	\qquad
	k \in \Z,
\end{equation}
where $x=z^{r}$ and $y=z^{s-r}$, as in the $(r,s)$ spectral curve, and the integral is understood as an asymptotic expansion for small absolute values of $\hbar$ near the critical point $\zeta = z$ (cf. \cite[Section~6]{ABDKS25-KP}), so that $\phi_k \in z^{k-1}\C\bbraket{z^{-s}\hbar}$. Concretely,
\begin{equation}
\begin{split}
	\phi_k(z;\hbar)
	&=
	\sqrt{(r-s)r z^{r-1}} e^{-\frac{1}{\hbar} \frac{r}{s} z^s}
	\int\
		\zeta^{k-1+\frac{s-r-1}{2}}
		e^{\frac{1}{\hbar} \bigl(z^r\zeta^{s-r} + \frac{r-s}{s}\zeta^s \bigr)}
		\frac{d\zeta}{\sqrt{2\pi\hbar}} \\
	&=
	z^{k-1}\bigl( 1 + \mathrm{O}(z^{-s}\hbar) \bigr).
\end{split}
\end{equation}
We will also need the dual\footnote{
	Compared to \cite{ABDKS25-KP}, we exchanged the role of the dual and relabeled the formal functions. More precisely: $\phi_k^{\textup{us}} = \phi_{1-k}^{\textup{them},*}$.
} formal functions defined as $\phi_k^*(z;\hbar) \coloneqq \phi_k(z;-\hbar)$.

\begin{definition}\label{d:wf}
	We define a normalized version of the formal functions, known as \emph{wave functions}, by setting
	\begin{equation}\label{eq:rs:Bessel:int}
		\psi_k(z;\hbar)
		\coloneqq
		\frac{e^{\frac{1}{\hbar} \frac{r}{s} z^s}}{\sqrt{rz^{r-1}}}
		\phi_k
		=
		\sqrt{(r-s)}
		\int\
			\zeta^{k-1+\frac{s-r-1}{2}}
			e^{\frac{1}{\hbar} \bigl(z^r\zeta^{s-r} + \frac{r-s}{s}\zeta^s \bigr)}
			\frac{d\zeta}{\sqrt{2\pi\hbar}},
	\end{equation}
	and analogously $\psi_k^*(z;\hbar) \coloneqq \psi_k(z;-\hbar)$. 
\end{definition}

Notice that the only difference between the functions $\phi_k$ and $\psi_k$ is in the so-called unstable terms, namely the value of the integrand at the critical point and the square-root of the second derivative of the exponential factor. The main reason for introducing such normalized versions is that the functions $\psi_k$ satisfy a very simple set of equations:
\begin{align}
	\label{eq:yhat}
	\hbar \frac{d}{dx} \psi_k &= \psi_{k+s-r}, \\
	\label{eq:xhat}
	x \psi_k &= \psi_{k+r} + \hbar \frac{2k+r-s-1}{2(r-s)} \psi_{k+r-s},
\end{align}
where again $x = z^r$. These equations can be recast as the following matrix differential system for the column vector of wave functions $\psi = (\psi_1,\ldots,\psi_r)^t$,
\begin{equation}\label{eq:rs:Bessel}
	\hbar \frac{d}{dx} \psi
	=
	\mc{D} \psi,
\end{equation}
where $\mc{D} = \mc{D}(x(z);\hbar)$ is the $r \times r$ matrix:
\begin{equation}\label{eq:calD}
	\mc{D}
	\coloneqq
	\begin{pNiceArray}{ccc|ccc}[margin]
		\frac{\hbar}{x} \frac{2+s-r-1}{2(r-s)} & & & \hphantom{0_{s}}\Block{3-3}{0_{(r-s) \times s}} & & \\
		& \ddots & & & & \\
		& & \frac{\hbar}{x} \frac{2(r-s)+s-r-1}{2(r-s)} & & & \\[.5ex]
		\hline
		\Block{3-3}{0_{s \times (r-s)}}
		& & & \Block{3-3}{0_{s \times s}} & & \\
		\\
		\\
	\end{pNiceArray}
	+
	\begin{pNiceArray}{ccc|ccc}[margin]
		\Block{3-3}{0_{(r-s) \times s}} & & & \frac{1}{x} & & \\
		& & & & \ddots & \\
		& & & & & \frac{1}{x} \\[.5ex]
		\hline
		1 & & & \Block{3-3}{0_{s \times (r-s)}} & & \\
		& \ddots & & \\
		& & 1 & \\
	\end{pNiceArray}.
\end{equation}
Here $0_{N\times M}$ denotes the zero matrix with $N$ rows and $M$ columns. Similar equations hold for $\psi_k^*$, after replacing $\hbar$ by $ -\hbar$.

\begin{example}
	As an example, the vector of wave functions for $(r,s) = (2,1)$ satisfies
	\begin{equation}
		\hbar \frac{d}{dx} \psi = \begin{pmatrix} 0 & 1/x \\ 1 & 0 \end{pmatrix}\psi.
	\end{equation}
	In other words,
	\begin{equation}
		\left(\hbar \frac{d}{dx} x \hbar \frac{d}{dx} - 1\right)\psi_1 = 0
		\qquad\quad\text{and}\qquad\quad
		\psi_2 = x \hbar \frac{d}{dx} \psi_1 \,.
	\end{equation}
	The first equation implies that $\psi_1$ is the (properly normalized, asymptotic expansion of the) solution to the Bessel differential equation, while the second equation expresses $\psi_2$ in terms of its derivative. Thus, the differential system \labelcref{eq:rs:Bessel} can be considered as a higher version of the Bessel differential equation, depending on the integer parameters $r \ge 2$ and $s \in [r-1]$.
\end{example}

We now introduce the Baker--Akhiezer kernel:

\begin{definition}\label{def:BA}
	For $z_1 \neq z_2$, we define the \emph{Baker--Akhiezer kernel} as
	\begin{equation}\label{eq:K:def}
		K(z_1,z_2;\hbar)
		\coloneqq
		\sum_{k \ge 1} \psi_{1-k}^*(z_1;\hbar) \psi_{k}(z_2;\hbar) \sqrt{dx_1 dx_2} ,
	\end{equation}
	where $x_1 = z_1^r$ and $x_2 = z_2^r$. The sum is understood formally as an element of
	\begin{equation}
		e^{-\frac{1}{\hbar} \frac{r}{s} (z_1^s - z_2^s)}
		\C
		\bbraket{\hbar z_1^{-s},\hbar z_2^{-s}}
		\bbraket{z_1^{-1},z_2}
		\sqrt{dz_1 dz_2} .
	\end{equation}
\end{definition}

The terminology is borrowed from the integrable-systems literature. In the algebro-geometric construction of KP solutions, Krichever introduced the distinguished wave function associated with spectral data, originally referring to it as the Akhiezer function~\cite{Kri77}; this object is now commonly called the Baker--Akhiezer function. In the Fock space formalism of KP, the corresponding wave and dual wave functions enter the bilinear identities for KP tau functions. The kernel above is the bilinear kernel built from such wave functions, and is therefore commonly referred to as the Baker--Akhiezer kernel \cite{ABDKS25-KP}.

We have the following alternative expression for the kernel as a \emph{finite} combination of wave functions, which will prove useful later on.

\begin{proposition}\label{prop:K:simpler}
	For $z_1 \neq z_2$, the Baker--Akhiezer kernel \eqref{eq:K:def} is given by
	\begin{equation}\label{eq:K:simpler}
		K(z_1,z_2;\hbar)
		=
		\frac{\sqrt{dx_1 dx_2}}{x_1 - x_2}
		\Biggl(
			\sum_{k = 1}^{r-s} x_1 \psi_{1-k}^*(z_1;\hbar)\psi_{k}(z_2;\hbar) 
			+
			\sum_{k=r-s+1}^{r}
				 \psi_{r+1-k}^*(z_1;\hbar) \psi_{k}(z_2;\hbar)
		\Biggr)
	\end{equation}
	where $x_i = z_i^r$.
\end{proposition}

\begin{proof}
	We omit the dependence on $\hbar$ for ease of notation. Multiplying the definition of the kernel by $(x_1 - x_2)$, applying \cref{eq:xhat}, and relabeling the indices, we find
	\begin{equation}
	\begin{split}
		(x_1 - x_2)K(z_1,z_2)
		&=
		\sum_{k \ge 1}
			\bigg[
				\left( \psi_{r+1-k}^*(z_1) - \hbar \frac{2(1-k)+r-s-1}{2(r-s)} \psi_{r-s+1-k}^*(z_1) \right)
				\psi_{k}(z_2) \\
		&\qquad -
			\psi_{1-k}^*(z_1)
			\left( \psi_{k+r}(z_2) + \hbar \frac{2k+r-s-1}{2(r-s)} \psi_{k+r-s}(z_2) \right)
		\bigg]
		\sqrt{dx_1 dx_2} \\
		&=
		\Bigg(
			\sum_{k = 1}^r
				\psi_{r+1-k}^*(z_1)\psi_{k}(z_2) \\
		&\qquad\qquad
		+ \hbar
			\sum_{k=1}^{r-s}
				\frac{2k-r+s-1}{2(r-s)} \psi_{r-s+1-k}^*(z_1) \psi_{k}(z_2)
		\Bigg) \sqrt{dx_1 dx_2} .
	\end{split}
	\end{equation}
	Finally, we use \cref{eq:xhat} to absorb the explicit $\hbar$-dependent terms.
\end{proof}

\begin{example}
	For $(r,s) = (2,1)$ the kernel reads
	\begin{equation}
		K(z_1,z_2;\hbar)
		=
		\frac{\sqrt{dx_1 dx_2}}{x_1 - x_2}
		\Biggl(
			\psi_{2}^*(z_1;\hbar)\psi_{1}(z_2;\hbar) 
			+
			\psi_{1}^*(z_1;\hbar) \psi_{2}(z_2;\hbar)
		\Biggr) ,
	\end{equation}
	where we have used $x \psi_0 = \psi_2$, which is a specialization of \cref{eq:xhat}. Since $\psi_1$ is a solution of the Bessel equation and $\psi_2 = x \hbar \frac{d}{dx} \psi_1$, it follows that the Baker--Akhiezer kernel for the $(2,1)$ spectral curve is precisely the Bessel kernel appearing in matrix models; see~\cite{EGGGL} for further details.
\end{example}

After resummation, $K(z_1,z_2;\hbar)$ is a WKB-type bispinor: after factoring out the exponential term $e^{-\frac{r}{s\hbar}(z_1^s-z_2^s)}$, it becomes a formal power series in $\hbar$ whose coefficients are meromorphic sections of \smash{$(K_{\mathbb{P}^1}^{1/2})^{\boxtimes 2}$}, with possible poles along $x_1=x_2$ and at $z_1=0$ or $z_2=0$:
\begin{equation}
	K(z_1,z_2;\hbar)
	\in
	e^{-\frac{1}{\hbar}\frac{r}{s}(z_1^s-z_2^s)} \cdot
	H^0\Bigl(
		(\mathbb{P}^1)^{\times 2},
		(K_{\mathbb{P}^1}^{1/2})^{\boxtimes 2}
		\bigl(\Delta + *D_{1} + *D_{2}\bigr)
	\Bigr)\bbraket{\hbar},
\end{equation}
where $\Delta=\set{x_1=x_2}$ and $D_i=\set{z_i=0}$. The expansion \eqref{eq:K:def} is retrieved by expanding this resummed expression in the sector $|z_1|>|z_2|$. One can show that the possible poles along $\Delta$ occur only in the coefficient of $\hbar^0$.

When $z_1 = z_2$, we will work with a regularized version of the kernel obtained by removing the singular part and adding the $(0,1)$ contribution. By a slight abuse of notation, we will continue to use the symbol $K(z,z;\hbar)$ for this regularized version. More precisely, we define
\begin{equation}
	K(z,z;\hbar)
	\coloneqq
	\lim_{z_0\to z} \left(
		e^{\frac{1}{\hbar}\frac{r}{s} (z^s-z_0^s)} K(z,z_0;\hbar)
		-
		\frac{\sqrt{dz dz_0}}{z - z_0}
	\right)
	-
	\frac{1}{\hbar} y(z) dx(z).
\end{equation}

\subsubsection{Another determinantal formula}
A restatement of \cite{ABDKS25-KP} for our specific choice of $x$ and $y$ gives a second determinantal formula for the correlators $\omega_{n}$.

\begin{proposition}\label{prop:det:K}
	The $n$-point correlators $\omega_n(z_{[n]};\hbar)$ produced by generalized topological recursion on the $(r,s)$ spectral curve are given by
	\begin{equation} \label{eq:det:K}
		\omega_{n}(z_{[n]};\hbar)
		=
		\sum_{\sigma \in C_n} \sgn(\sigma) \prod_{i=1}^n K(z_i,z_{\sigma(i)};\hbar) .
	\end{equation}
\end{proposition}

In the formula above, the products of kernels are understood using the resummed representatives of \cref{prop:K:simpler}, rather than the one-sided formal expansion in \cref{eq:K:def}. With this interpretation the products are well-defined. The determinantal formula then implies several non-trivial cancellations: for each cycle $\sigma\in C_n$, the WKB exponential factors cancel; the apparent poles along $x_i=x_j$ cancel in the cyclic sum; and the resulting expansion is of topological type. Hence the right-hand side is a formal series in $\hbar$ whose coefficients are meromorphic sections of $K_{\mathbb P^1}^{\boxtimes n}$, with possible poles only at $z_i=0$, as expected from the left-hand side.

\begin{proof}
	In~\cite[Section~3.2]{ABDKS25-KP}, the result is expressed in terms of the following variant of the kernel
	\begin{equation}
		\tilde{K}(z_1,z_2;\hbar) = \sum_{k \ge 1} \phi_{1-k}^*(z_1;\hbar) \phi_{k}(z_2;\hbar) \sqrt{dz_1 dz_2} .
	\end{equation}
	Specifically, the result reads
	\begin{equation}
	\begin{aligned}
		\omega_{1}(z;\hbar)
		&=
		\lim_{z_0\to z} \left(
			\tilde{K}(z,z_0;\hbar)
			-
			\frac{\sqrt{dz dz_0}}{z - z_0}
		\right) - \frac{1}{\hbar} y(z) dx(z),
		\\
		\omega_{n}(z_{[n]};\hbar)
		&=
		\sum_{\sigma \in C_n} \sgn(\sigma) \prod_{i=1}^n \tilde{K}(z_i,z_{\sigma(i)};\hbar),
		\qquad\qquad\qquad\qquad
		n \ge 2.
	\end{aligned}
	\end{equation}
	The difference between $K$ and $\tilde{K}$ is simply in the exponential prefactor:
	\begin{equation}
		K(z_1,z_2;\hbar) = e^{-\frac{1}{\hbar}\frac{r}{s} (z_1^s-z_2^s)} \tilde{K}(z_1,z_2;\hbar).
	\end{equation}
	Thus, we see that for $n = 1$ the formula follows from the definition of the kernel $K(z,z;\hbar)$ by regularization. In the formula for $n \geq 2$, all the exponential factors cancel out, as each variable $z_i$ appears precisely twice in \cref{eq:det:K}: once contributing a factor of $e^{\frac{1}{\hbar}\frac{r}{s} z_i^s} $ and the other time contributing a factor of $e^{-\frac{1}{\hbar}\frac{r}{s} z_i^s} $, which cancel out.
\end{proof}

\subsubsection{Matrix forms}
In light of the differential system \labelcref{eq:rs:Bessel}, it is natural to consider the following matrix solution, called the wave matrix:
\begin{equation}
	\Psi(z;\hbar)
	=
	\bigl( \Psi_{k,a}(z;\hbar) \bigr)_{k,a=1,\ldots,r}
	\coloneqq
	\bigl( \psi_k(\theta^a z;\hbar) \bigr)_{k,a=1,\ldots,r}
\end{equation}
where $\theta = e^{\frac{2\pi \iu}{r}}$ is a primitive $r$-th root of unity. By design, the wave matrix satisfies $\hbar \frac{d}{dx} \Psi = \mc{D} \Psi$ and its last column coincides with the vector of wave functions $\psi$ considered in the previous section. Similarly, we define $\Psi^*$ from the dual wave functions.

A simple computation implies that the inverse wave matrix satisfies $-\hbar \frac{d}{dx} \Psi^{-1} = \Psi^{-1}\mc{D}$, which looks like the same differential equation with the sign of $\hbar$ swapped. This would imply a simple relation between the inverse wave matrix and its dual. However, this is not immediately clear, as the matrix $\mc{D}$ does depend on $\hbar$. Nonetheless, the following relation still holds.

\begin{lemma}
	The inverse wave matrix is given by
	\begin{equation}\label{eq:Psi:inv}
		\Psi^{-1}_{a,k}(z;\hbar)
		=
		\begin{cases}
			x \psi^*_{1-k}(\theta^a z;\hbar) & 1 \le k \le r-s, \\[1ex]
			\psi^*_{r+1-k}(\theta^a z;\hbar) & r-s < k \le r.
		\end{cases}
	\end{equation}
\end{lemma}

\begin{proof}
	Denote as $\tilde{\Psi}_{a,k}(z;\hbar)$ the right-hand side of \cref{eq:Psi:inv}. A direct computation shows that the matrix $\tilde{\Psi}$ satisfies the same differential equation as that satisfied by $\Psi^{-1}$. Hence, they coincide up to a normalization factor. This factor can be fixed by looking at the first coefficient in the $\hbar$-expansion of both matrices:
	\begin{equation}
		\Psi_{k,a}(z;\hbar)
		=
		\theta^{a(k-1 - \frac{r-1}{2})}
		\frac{e^{\frac{1}{\hbar}\frac{r}{s} \theta^{as} z^s}}{\sqrt{rz^{r-1}}}
		z^{k-1}
		\bigl(
			1 + \bigO(\hbar)
		\bigr)
	\end{equation}
	which implies that
	\begin{equation}
		\Psi^{-1}_{a,k}(z;\hbar)
		=
		\theta^{-a(k-1 - \frac{r-1}{2})}
		\frac{e^{-\frac{1}{\hbar}\frac{r}{s} \theta^{as} z^s}}{\sqrt{r z^{r-1}}}
		z^{r-k}
		\bigl(
			1 + \bigO(\hbar)
		\bigr).
	\end{equation}
	One can check that the same asymptotics holds for $\tilde{\Psi}_{a,k}$, which concludes the proof.
\end{proof}

As we did for the wave function, it will prove useful to introduce the following matrix version of the kernel:
\begin{equation}
	\mc{K}(z_1,z_2;\hbar)
	=
	\left( \mc{K}_{a,b}(z_1,z_2;\hbar)\right)_{a,b=1,\ldots,r}
	\coloneqq
	\left(K(\theta^a z_1,\theta^b z_2;\hbar)\right)_{a,b=1,\ldots,r} .
\end{equation}
Notice that there is no ambiguity in the substitution coming from the square roots. Indeed, the kernel is understood as a spinor in the $z_i$-variables: the factor $\sqrt{r z_i^{r-1}}$ in $\sqrt{dx_i}$ cancels with the same factor in the normalization of the wave functions. Thus, after this cancellation, the only remaining transformation is the standard spinor one, $\sqrt{d(\theta^{a_i}z_i)}=\theta^{a_i/2}\sqrt{dz_i}$, with the branch fixed once and for all.

We now show that $\mc{K}(z_1,z_2;\hbar)$ can be compactly written in terms of the wave matrix as follows.

\begin{lemma}\label{lem:matK}
	The matrix kernel $\mc{K}$ can be expressed as
	\begin{equation}\label{eq:matK}
		\mc{K}(z_1,z_2;\hbar)
		=
		\begin{cases}
			- \frac{1}{\hbar} \Psi^{-1}(z_1;\hbar)\mc{D}(x_1;\hbar)\Psi(z_1;\hbar) \, dx_1,
			& \text{if }z_1 = z_2, \\[1ex]
			\frac{\Psi^{-1}(z_1;\hbar)\Psi(z_2;\hbar)}{x_1 - x_2} \sqrt{dx_1 dx_2},
			& \text{if }z_1 \neq z_2, \\
		\end{cases}
	\end{equation}
	where as usual $x_i = z_i^r$.
\end{lemma}

\begin{proof}
	The second equation follows from rewriting the result of \cref{prop:K:simpler} using the expression obtained in \cref{eq:Psi:inv} for $\Psi^{-1}(z;\hbar)$. As for the case when $z_1=z_2$, we use L'Hôpital's rule to calculate the limit. When $a\neq b$:
	\begin{equation}
		\mc{K}_{a,b}(z,z;\hbar)
		= 
		\lim_{z_0\to z} \frac{\left( \Psi^{-1}(z;\hbar)\Psi(z_0;\hbar)\right)_{a,b}} {x-x_0} \sqrt{dx dx_0}
		=
		-\frac{1}{\hbar} \left( \Psi^{-1}(z;\hbar)\mc{D}(x;\hbar)\Psi(z;\hbar)\right)_{a,b} dx,
	\end{equation}
	where we use the matrix differential system $ \hbar \frac{d}{dx} \Psi = \mc{D}\Psi$. The case when $a=b$ is slightly more involved:
	\begin{multline}
		\mc{K}_{a,a}(z,z;\hbar)
		=
		- \frac{1}{\hbar} y(\theta^a z) dx + \\
		+
		\lim_{z_0\to z} \left(
			e^{\frac{1}{\hbar}\frac{r}{s} \theta^{as}(z^s - z_0^s)}
			\frac{\left( \Psi^{-1}(z;\hbar)\Psi(z_0;\hbar)\right)_{a,a}}{x - x_0}
			\sqrt{dx dx_0}
			-
			\frac{\sqrt{dz dz_0}}{z - z_0}
		\right) . 	
	\end{multline} 
	Now, applying L'Hôpital's rule produces an extra summand due to the exponential prefactor, which precisely cancels the $y dx$ term in the first line to give the result.
\end{proof}

\subsection{Towards the loop equations}
In this section, we use the determinantal formulas for the $n$-point correlators $\omega_n$ derived in \cref{prop:det:K} to express a certain combination of these correlators as a characteristic polynomial. We essentially follow the approach of~\cite{BEM18}, adapted to our (simpler) setting. This, in turn, allows us to extract conditions on the poles of such combinations in \cref{ssec:loop:eq}, which are precisely the loop equations.

\subsubsection{Some definitions}
First, let us introduce the following $(k+n)$-differentials:
\begin{equation} \label{eq:omega-k-disc-n-conn}
	\omega_{k;n}(z_{[k]};w_{[n]};\hbar)
	\coloneqq
	\sum_{\substack{ \bm{L} \vdash [k] \\ \bigsqcup_{L \in \bm{L}} M_L = [n] } }
		\prod_{L\in \bm{L}} \omega_{|L|+|M_L|} (z_L,w_{M_L};\hbar) .
\end{equation}
They represent $k$-disconnected, $n$-connected correlators, or semi-connected correlators for short. For a graphical explanation of what these objects are, see~\cite[Section~2.2]{BBCKS25}.

Then, for $k \in \set{0,\ldots,r}$, we define the correlator $\mc{E}^{(k)}_{n}(z;z_{[n]};\hbar)$ as
\begin{equation}
	\mc{E}^{(k)}_{n}(z;z_{[n]};\hbar)
	\coloneqq
	\sum_{\substack{Z \subseteq x^{\leftarrow}(z) \\ |Z|=k}} \omega_{k;n}(Z;z_{[n]};\hbar),
\end{equation}
where by convention $\mc{E}^{(0)}_{n}(z;z_{[n]};\hbar) = \delta_{n,0}$ and
\begin{equation}
	x^{\leftarrow}(z) \coloneqq x^{-1}(x(z)) = \Set{ z' | x(z') = x(z) }
\end{equation}
is the set of points of the spectral curve mapping to the same point as $z$. Notice that $\mc{E}^{(k)}_{n}(z;z_{[n]};\hbar)$ is a $k$-differential in $z$, and a differential in $z_1, \ldots, z_n$.

The correlators \smash{$\mc{E}^{(k)}_{n}$} are the main characters appearing in the abstract loop equations of the Bouchard--Eynard topological recursion, see \cite{BE17, BBCKS25}. As we are summing over all preimages of the point $x(z)$, the differential \smash{$\mc{E}^{(k)}_{n}$} is actually the pullback of a $k$-differential in $x(z)$ from the base to the spectral curve. Furthermore, notice that \smash{$\mc{E}^{(k)}_{n}$} has the following genus expansion:
\begin{equation}
	\mc{E}^{(k)}_{n}(z;z_{[n]};\hbar)
	=
	\sum_{g \geq 0} \hbar^{2g-k+n} \mc{E}^{(k)}_{g,n}(z;z_{[n]}),
\end{equation}
where the differentials $\mc{E}^{(k)}_{g,n}(z;z_{[n]})$ do not depend on $\hbar$ and coincide with the ones defined in \cite[Section 2.2]{BBCKS25}.

\subsubsection{A characteristic polynomial equation}
Our goal is to express the differentials $\mc{E}^{(k)}_{n}$ as the coefficients of a characteristic polynomial. In order to do so, we first need to realize the semi-connected correlators $\omega_{k;n}$ as a determinant-like expression.

To this end, define $S_{k;n} \subseteq S_{k+n}$ to be the subset of all permutations $\tau$ in $S_{k+n}$ on the elements $\set{1,2,\ldots,k+n}$ that do not fix any non-empty subset of $\set{k+1,k+2,\ldots,k+n}$. In other words, when $\tau$ is expressed as a product of disjoint cycles, each cycle that appears contains at least one element from the set $\set{1,2,\ldots,k}$.

\begin{lemma}
	Let $Z \subseteq x^{\leftarrow}(z)$ be a subset of cardinality $|Z| = k$. Choose an arbitrary ordering of $Z$ so that it can be written as the $k$-tuple $Z = (\theta^{a_1} z, \ldots, \theta^{a_k} z)$ for some pairwise distinct integers $a_1, \ldots, a_k \in [r]$. Then, we have
	\begin{equation}\label{eq:WknusingK}
		\omega_{k;n}(Z;z_{[n]};\hbar)
		=
		\sum_{\tau \in S_{k;n}} \sgn(\tau) \prod_{i=1}^{k+n} K(\hat{z}_i,\hat{z}_{\tau(i)};\hbar),
	\end{equation}
	where we have set
	\begin{equation}\label{eq:zhat}
		\hat{z}_i
		\coloneqq
		\begin{cases}
			\theta^{a_i} z & 1 \leq i \leq k, \\
			z_{i-k} & k < i \leq k+n .
		\end{cases}
	\end{equation}
\end{lemma}

\begin{proof}
	First of all, note that both sides of \eqref{eq:WknusingK} do not depend on the choice of ordering of the elements in $Z$, since the first $k$ elements are treated symmetrically in both $\omega_{k;n}$ and $S_{k;n}$.

	Here and below, for $I\subseteq[n]$, we write $k+I=\set{k+i\mid i\in I}\subseteq\set{k+1,\ldots,k+n}$. Now recall the definition~\eqref{eq:omega-k-disc-n-conn} of the semi-connected correlators, and apply the determinantal formula~\eqref{eq:det:K} to each factor $\omega_{|L|+|M_L|}(Z_L,z_{M_L};\hbar)$, where $L$ is a block of the partition $\bm L$ of $[k]$. Each such factor is then a sum over cycles $\tau_L$ with support exactly $L\sqcup(k+M_L)$, namely cycles of length $|L|+|M_L|$; in particular $\sgn(\tau_L)=(-1)^{|L|+|M_L|-1}$. 
	Thus, we find
	\begin{equation}\label{eq:omega-k-conn-n-disc} 
		\omega_{k;n}(Z; z_{[n]}; \hbar) 
		=
		\sum_{ \substack{ \bm{L} \vdash [k] \\ \bigsqcup_{L \in \bm{L}} M_L = [n] } }
		\sum_{\tau_L}
			\left( \prod_{L \in \bm{L}} \sgn(\tau_L) \right) 
			\left(
				\prod_{L \in \bm{L}} \prod_{i \in L \sqcup (k+M_L)}
				K(\hat{z}_i, \hat{z}_{\tau_L(i)}; \hbar)
			\right).
	\end{equation}
	Here the second sum is over all choices of such a cycle $\tau_L$ for every block $L\in\bm L$.

	Note that, by definition of $S_{k;n}$, each permutation $\tau \in S_{k;n}$ can be uniquely expressed as a product of disjoint cycles $\prod_{L \in \bm{L}} \tau_L$, where $\bm L$ is the partition of $[k]$ whose blocks record which elements of $[k]$ lie in each cycle of $\tau$. For each block $L\in\bm L$, the corresponding cycle $\tau_L$ has support $L\sqcup(k+M_L)$ for a uniquely determined subset $M_L\subseteq[n]$, and the subsets $M_L$ form a disjoint partition of $[n]$. Moreover, $\sgn(\tau)=\prod_{L\in\bm L}\sgn(\tau_L)$. In other words, the right-hand side of~\eqref{eq:omega-k-conn-n-disc} is simply a reformulation of the right-hand side of~\eqref{eq:WknusingK}, in which the sum over $\tau\in S_{k;n}$ is reorganized as a sum over the disjoint cycles of $\tau$.
\end{proof}

We can now realize the semi-connected correlators $\omega_{k;n}$ as a principal minor of size $k$ of an $r \times r$ matrix, that is as a sum over the symmetric group $S_k$ instead of $S_{k;n} \subseteq S_{k+n}$.

\begin{lemma} \label{lem:WusingSk}
	Let $Z \subseteq x^{\leftarrow}(z)$ be a subset of cardinality $|Z| = k$, arbitrarily ordered as $Z = ( \theta^{a_1}z, \ldots, \theta^{a_k}z )$. Then, we have 
	\begin{equation}\label{eq:ESk}
		\omega_{k;n}(Z;z_{[n]};\hbar)
		=
		(-1)^n
		[\epsilon_1\cdots\epsilon_n]
			\sum_{\sigma \in S_k}
				\sgn(\sigma)
				\prod_{j=1}^k
					\Bigl( \mc{J}_n(z;z_{[n]};\hbar,\epsilon) \Bigr)_{a_j,a_{\sigma(j)}} ,
	\end{equation}
	where $\mc{J}_n$ is the $r \times r$ matrix defined as
	\begin{multline}
		\mc{J}_n(z;z_{[n]};\hbar,\epsilon)
		\coloneqq \mc{K}(z,z;\hbar) + 
		\sum_{\ell = 1}^n \sum_{1 \le i_1 \neq \cdots \neq i_{\ell} \le n}
			\epsilon_{i_1} \cdots \epsilon_{i_\ell}
			\times \\
			\times
			\mc{K}(z,z_{i_1};\hbar)
			E_r
			\mc{K}(z_{i_1},z_{i_2};\hbar)
			E_r
			\cdots
			E_r
			\mc{K}(z_{i_{\ell-1}},z_{i_\ell};\hbar)
			E_r
			\mc{K}(z_{i_\ell},z;\hbar)
	\end{multline}
	and $E_r$ is the $r \times r$ elementary matrix whose only non-zero entry is $(E_r)_{r,r} = 1$. 
\end{lemma}

\begin{proof}
	We drop the $\hbar$ for ease of notation. Let us rewrite the right-hand side of \eqref{eq:ESk} by extracting the coefficient of $\epsilon_1\cdots\epsilon_n$:
	\begin{equation}\label{eq:WRHS}
		(-1)^n \sum_{\sigma \in S_k}
			\sgn(\sigma) 
			\sum_{\substack{I_1\sqcup \cdots \sqcup I_k = [n] \\ I_j = (i^{(j)}_1,\cdots, i^{(j)}_{|I_j|})}}
				\prod_{j=1}^k
				\mc{K}_{a_j,r}(z,z_{i^{(j)}_1})
				\mc{K}_{r,r}(z_{i^{(j)}_1},z_{i^{(j)}_2})
				\cdots
				\mc{K}_{r,a_{\sigma(j)}}(z_{i^{(j)}_{|I_j|}},z) .
	\end{equation}
	In the innermost sum, each $I_j$ for $j \in [k]$ is a possibly empty ordered set such that the disjoint union of all the $I_j$ equals $[n]$. If a certain set $I_j$ is empty, the corresponding product of kernels consists of the single term $\mc{K}_{a_j,a_{\sigma(j)}}(z,z)$. 

	From a permutation $\sigma \in S_k$ (viewed as a permutation in $S_{k+n}$) and the sets $\left(I_j\right)_{j\in[k]}$ appearing in \cref{eq:WRHS}, we can build the following permutation $\tau \in S_{k+n}$:
	\begin{equation}
		\tau
		=
		\sigma
		\circ
		\left( 1 , k+ i_1^{(1)} , \ldots , k+ i^{(1)}_{|I_1|} \right)
		\circ \cdots \circ
		\left( k , k+ i_1^{(k)} , \ldots , k+ i^{(k)}_{|I_k|} \right).
	\end{equation}
	In fact, the above permutation clearly lives in $S_{k;n} \subseteq S_{k+n} $, and every permutation in $S_{k;n}$ can uniquely be expressed in the above form. Using this bijection, and noting that $\sgn(\tau) = (-1)^{n}\sgn(\sigma)$, we can rewrite \cref{eq:WRHS} as 
	\begin{equation}
		\sum_{\tau \in S_{k;n} \subseteq S_{k+n}}
			\sgn(\tau) \prod_{i=1}^{k+n} K(\hat{z}_i,\hat{z}_{\tau(i)}),
	\end{equation}
	where we use the notation from \cref{eq:zhat} for $\hat{z}_i$. This concludes the proof.
\end{proof}

In light of the above lemma and \cref{eq:matK} expressing the kernel in terms of the wave matrix, it is natural to introduce\footnote{
	Lie-theoretically, $\Psi$ can be viewed as a flat section of the trivial bundle over $\P^1$ with connection $\nabla_{\hbar} = \hbar d - \mc{D} dx$. Then, $M$ is a flat section of the adjoint bundle. This is the point of view adopted in \cite{BEM18}.
}
\begin{equation}
	M(z;\hbar)
	\coloneqq
	\Psi(z;\hbar) E_r \Psi^{-1}(z;\hbar) \, dx.
\end{equation}
Again, $E_r$ denotes the $r \times r$ elementary matrix whose only non-zero entry is $(E_r)_{r,r} = 1$. We also define, following the usual convention $x = z^r$ and $x_i = z_i^r$,
\begin{equation}\label{eq:calAdef}
	\mc{A}_n(z;z_{[n]};\hbar,\epsilon)
	\coloneqq \frac{1}{\hbar} \mc{D}(x;\hbar) \, dx
	-
	\sum_{\ell = 1}^n \sum_{1 \le i_1 \neq \cdots \neq i_{\ell} \le n}
		\epsilon_{i_1} \cdots \epsilon_{i_\ell}
			\frac{
				M(z_{i_1};\hbar)\cdots M(z_{i_\ell};\hbar) \, dx
			}{
				(x-x_{i_1})(x_{i_1}-x_{i_2}) \cdots (x_{i_\ell}-x)
			}.
\end{equation}
Then, we have the following result which is a special case of~\cite[Theorem~4.3]{BEM18}.

\begin{proposition}
	We have the following formula
	\begin{equation}\label{eq:Easdet}
		\sum_{k=0}^r \left(y dx(z)\right)^{r-k} \mc{E}^{(k)}_n (z;z_{[n]};\hbar)
		=
		(-1)^{n} [\epsilon_1\cdots \epsilon_n]
		\det{\Bigl(
			y dx(z) \Id_{r\times r}
			-
			\mc{A}_n(z;z_{[n]};\hbar,\epsilon)
		\Bigr)} .
	\end{equation}
\end{proposition}

\begin{proof}
	We omit the dependence on $\hbar$ and $\epsilon$ from the various functions for ease of notation. Using the expression for $\omega_{k;n}$ proved in \cref{lem:WusingSk}, we can write 
	\begin{equation}
		\mc{E}^{(k)}_n (z;z_{[n]})
		=
		(-1)^n [\epsilon_1\cdots\epsilon_n]
		\sum_{\substack{Z \subseteq x^{\leftarrow}(z) \\ Z = \set{\theta^{a_1}z,\ldots, \theta^{a_k}z}}}
			\sum_{\sigma \in S_k}
				\sgn(\sigma)
				\prod_{j=1}^k
						\Bigl(
							\mc{J}_n(z;z_{[n]})
						\Bigr)_{a_j,a_{\sigma(j)}} .
	\end{equation}
	Combining the definitions of the matrices $\mc{J}_{n}$ and $\mc{A}_{n}$, \cref{eq:matK,eq:calAdef}, with the expression of the matrix kernels in terms of the wave matrices, \cref{lem:matK}, we can write the sum over the symmetric group as 
	\begin{equation}
		(-1)^k
		\sum_{\sigma \in S_k}
			\sgn(\sigma)
			\prod_{j=1}^k
				\Bigl(	
					\Psi^{-1}(z)
					\mc{A}_n(z;z_{[n]})
					\Psi(z)
				\Bigr)_{a_j,a_{\sigma(j)}} .
	\end{equation}
	This can be viewed as a principal minor of size $k$ of the $r \times r$ matrix $\Psi^{-1}\mc{A}_n\Psi$, and the sum over $Z$ is the sum over all such principal minors. As the coefficients of the characteristic polynomial can be expressed as a weighted sum of all principal minors, we find
	\begin{equation}
		\sum_{k=0}^r \left(y dx(z)\right)^{r-k} \mc{E}^{(k)}_n (z;z_{[n]})
		=
		(-1)^{n} [\epsilon_1\cdots \epsilon_n]
		\det{\Bigl(
			y dx(z) \Id_{r\times r}
			-
			\Psi^{-1}(z)\mc{A}_n(z;z_{[n]})\Psi(z)
		\Bigr)} .
	\end{equation}
	The invariance of the determinant under conjugation implies the statement.
\end{proof}

\subsection{Loop equations}
\label{ssec:loop:eq}
We are finally ready to derive the loop equations, which is a statement about the behavior of the correlators $\mathcal{E}^{(k)}_n$ as $x \to 0 $. Recall that the $\mathcal{E}^{(k)}_n(z;z_{[n]};\hbar)$ are pullbacks of $k$-differentials in $x = z^r$. 

\begin{theorem}[Loop equations]\label{thm:loop:eq}
	As $x \to 0$, the differentials $\mc{E}_n^{(k)}(z;z_{[n]};\hbar)$ behave as
	\begin{equation}\label{eq:loop}
		\mc{E}_n^{(k)}(z;z_{[n]};\hbar)
		=
		\begin{cases}
			\delta_{n,0} A_k \frac{dx^k}{x^{k}} + \bigO(\frac{dx^k}{x^{k-1}}) & 1 \leq k \leq r-s, \\[1ex]
			\bigO(\frac{dx^k}{x^{r-s}}) & r-s < k \leq r,
		\end{cases}
	\end{equation}
	where the constants $A_k$ are given in terms of the elementary symmetric polynomial $e_k$ by
	\begin{equation}\label{eq:Akdef}
		A_k
		\coloneqq
		e_k\left(
			\frac{2+s-r-1}{2(r-s)}, \frac{4+ s-r-1}{2(r-s)}, \ldots, \frac{2(r-s) + s-r-1}{2(r-s)}
		\right).
	\end{equation}
	It is worth noting that the constants $A_k$ vanish unless $k$ is even.
\end{theorem}

\begin{proof}
	We begin with \eqref{eq:Easdet} for the differentials $\mathcal{E}^{(k)}_n$. Write 
	\begin{equation}
		\mathcal{P} dx
		\coloneqq
		y dx \, \Id_{r\times r} - \mc{A}_n
		=
		y dx \, \Id_{r\times r} - \frac{1}{\hbar} \, \mc{D} dx
		- \left( \mathcal{A}_n - \frac{1}{\hbar} \, \mc{D} dx \right).
	\end{equation}
	The goal is to understand the terms of highest order in $1/x$, i.e. the most singular ones as $x \to 0$, contributing to the determinant of $\mathcal{P} dx$. Note that all the entries of $\mathcal{A}_n - \frac{1}{\hbar} \, \mc{D} dx$ are regular as $x \to 0$. From the form of $\mathcal{D}$ in \cref{eq:calD}, we see that there are singular terms appearing on the main diagonal and the $(s+1)$-th upper diagonal. Thus, all the singular contributions come from these elements of $\mathcal{D}$.
	
	The key observation is to note that in the Laplace expansion for the determinant, the following two types of permutations contribute the highest powers of $1/x$ to the determinant.
	\begin{itemize}
		\item
		For any $p \in [s]$, consider the permutation
		\begin{equation}
			\sigma_p
			\coloneqq
			\left(
				p , (s+p) , (2s+p) , \ldots , \left\lfloor \frac{r-p}{s} \right\rfloor s + p
			\right) .
		\end{equation}
		Then, in the Laplace expansion of $\det \mathcal{P}dx$, we pick $\left\lfloor \frac{r-p}{s} \right\rfloor $ of the $1/x$ terms appearing on the $(s+1)$-th upper diagonal of $\mc{D}$. More precisely, $\sigma_p$ contributes 
		\begin{equation}
			\mathcal{P}_{\left\lfloor \frac{r-p}{s} \right\rfloor s + p,p }
			\prod_{b=1}^{\left\lfloor \frac{r-p}{s} \right\rfloor }	\mathcal{P}_{(b-1)s + p,bs+p}
			=
			\alpha \, x^{-\left\lfloor \frac{r-p}{s} \right\rfloor}
			+
			\bigO\bigl( x^{-\left\lfloor \frac{r-p}{s} \right\rfloor +1} \bigr)
		\end{equation}
		where $\alpha$ is a constant independent of $x$ and $y$.
		
		\item
		For any $p \in [s]$, consider the permutation $\tau_p$ which acts as the identity on all the integers that appear in $\sigma_p$. Then the contribution from $\tau_p$ is of the form
		\begin{equation}
			\prod_{b=0}^{\left\lfloor \frac{r-p}{s} \right\rfloor }
			\mathcal{P}_{bs + p,bs+p}
			=
			\beta \, y \, x^{-\left\lfloor \frac{r-p}{s} \right\rfloor}
			+
			\bigO\bigl( x^{-\left\lfloor \frac{r-p}{s} \right\rfloor+1} \bigr),
		\end{equation}
		where $\beta$ is a constant independent of $x$ and $y$, and the terms in $\bigO(x^{-\left\lfloor \frac{r-p}{s} \right\rfloor+1})$ may depend on $y$. In this case, in the Laplace expansion, we pick the $1/x$ terms appearing in the main diagonal of $ \mc{D} $ for all the terms except the last. The last term is chosen to be $y$ as $ \left\lfloor \frac{r-p}{s} \right\rfloor s + p > r-s$, and thus, we cannot choose a term of the form $1/x$.
	\end{itemize}
	
	The most singular terms in the determinant are then produced by permutations that are products of the permutations $\sigma_p$ and $\tau_p$ for $p \in [s]$. In other words, we need to choose a permutation $\sigma_p$ or $\tau_p$ for every $p \in [s]$. Depending on the number of times we choose $\sigma_p$, say $0 \leq t \leq s$ times, the most singular term is 
	\begin{equation}
		\gamma \, x^{-\sum_{p=1}^{s} \left\lfloor \frac{r-p}{s} \right\rfloor} y^{s-t}
		=
		\gamma \, x^{s-r} y^{s-t} ,
	\end{equation}
	where $\gamma$ is a constant independent of $x$ and $y$. This proves the statement for $r-s < k \leq r$ in \cref{eq:loop}.
	
	Consider now the case $1 \leq k \leq r-s$. In order to get terms that contain $y^{s+t}$ for $0 \leq t \leq r-s-1 $, we need to consider a permutation $\sigma$ in $S_r$ that acts as the identity on at least $s+t$ elements. Our previous argument shows that the most singular contributions that also contain the highest powers of $y$ are obtained by choosing the permutation $\tau_p$ for every $p \in [s]$. Thus the most singular terms we get are given by
	\begin{equation}
		y^s \prod_{i=1}^{r-s} \left( y + \frac{1}{x} \frac{2i + s -r -1}{2(r-s)} \right) = \sum_{t=0}^{r-s} y^{s+t} A_{r-s-t}\frac{1}{x^{r-s-t}} ,
	\end{equation}
	where $A_k$ is defined in \cref{eq:Akdef}. Notice that, as the above terms are independent of $\epsilon_{[n]}$, they only contribute to the correlator \smash{$\mathcal{E}^{(k)}_n$} when $n = 0$ due to \cref{eq:Easdet}.
	
	When $n >0$, we are forced to choose a term $(\mathcal{A}_n - \frac{1}{\hbar} \, \mc{D} dx)_{c,c}$ for some $1 \leq c \leq r-s$ on the diagonal instead of a term from $\mc{D}_{c,c}$ that contributes $1/x$. Thus the most singular contribution with the power $y^{s+t}$ is
	\begin{equation}
		\delta \, y^{s+t} \left( \frac{1}{x}\right)^{r-s-t-1}
	\end{equation}
	for some $0 \leq t \leq r-s-1$, and some constant $\delta$ independent of $x$ and $y$. This concludes the proof.
\end{proof}

\subsection{Solving the loop equations}
The loop equations of \cref{thm:loop:eq} do not admit a unique solution in general, as will become clear in \cref{sec:W:constraints} when we discuss the associated $\mathcal{W}$-constraints. However, for certain special values of $s \in [r-1]$, they do yield a unique solution.

Let us introduce some notation before stating the results. Given correlators $\omega_{g,n}$, we define the combination $\omega'_{g,k;n}$ by
\begin{equation}
	\omega'_{g,k;n} (z_{[k]},w_{[n]})
	\coloneqq
	\sum'_{\substack{ \bm{L} \vdash [k] \\ \bigsqcup_{L \in \bm{L}} M_L = [n] \\ \sum_{L \in \bm{L}} (g_L-1) = g - k} }
	\prod_{L\in \bm{L}} \omega_{g_L,|L|+|M_L|} (z_L,w_{M_L}) ,
\end{equation}
where the prime on the sum indicates that we omit all terms in which $\omega_{0,1}$ appears. These are the genus-$g$ semi-connected correlators without disks. For a graphical interpretation of this expression, see~\cite[Section~2.2]{BBCKS25}.

\subsubsection{The case $s = r-1$}
When $s = r-1$, the loop equations of \cref{thm:loop:eq} have a unique solution, which is given by the Bouchard--Eynard topological recursion on the $(r,r-1)$ spectral curve (recall  \cref{def:rs} of the $(r,s)$ spectral curve). This provides an alternate proof of the following result of~\cite{CGG25}. 

\begin{proposition}
	The Bouchard--Eynard topological recursion correlators $\omega_{g,n}$ on the $(r,r-1)$ spectral curve given by $x = z^r$ and $y = z^{-1}$ (with $\omega_{0,1} = -y dx$) are generating functions for the descendant integrals of the $\Theta^{r,r-1}$-classes. More precisely, for $2g-2+n>0$,
	\begin{equation}
		\omega_{g,n}(z_{[n]})
		=
		\left( \frac{1}{r}  \right)^{2g-2+n}
		\sum_{\substack{ k_1,\dots,k_n \ge 0 \\ a_1,\dots,a_n \in [r-1] }}
		\int_{\overline{\mathcal{M}}_{g,n}}
			\Theta^{r,r-1}_{g,n}(a) \prod_{i=1}^n \psi_i^{k_i} d \xi_{k_i,a_i}(z_i)
	\end{equation}
	with $d \xi_{k,a}(z) = (rk+a)!^{(r)}\frac{dz}{z^{rk+a+1}}$.
\end{proposition}

\begin{proof}
	After expanding in $\hbar$, the loop equations of \cref{thm:loop:eq} read, for any $1 \leq k \leq r$,
	\begin{equation}
		\mathcal{E}^{(k)}_{g,n} (z;z_{[n]})
		=
		\bigO\biggl( \frac{dx^k}{x^{1-\delta_{k,1}}} \biggr)
	\end{equation}
	as $A_1 = 0$. Thus, the loop equations of \cref{thm:loop:eq} coincide with the loop equations found in \cite{BBCCN24} for the Bouchard--Eynard topological recursion on the $(r,r-1)$ spectral curve. Appendix~C of \textit{loc. cit.} proves that the solution to these loop equations is unique. Combining this with \cref{thm:WgnTheta} gives the statement of the proposition.
\end{proof}

\begin{remark}
	In fact, one can obtain this proposition directly as a consequence of \cref{thm:gTR}. Indeed, when $s = r-1$, by taking limits one can show that the generalized topological recursion of~\cite{ABDKS25-gTR} coincides with the Bouchard--Eynard topological recursion of~\cite{BE13,BHLMR14}. See~\cite[Section~5.4]{ABDKS25-gTR} for more details.
\end{remark}

\subsubsection{The case $s = 1$}
When $s=1$, a more general version of the loop equations of \cref{thm:loop:eq} has been studied in \cite{BBKN25} under the name of shifted loop equations. The authors also prove that the unique solution to the shifted loop equations is constructed by the so-called shifted topological recursion. Combining this result with \cref{thm:WgnTheta} gives the following result. 

\begin{proposition}
	The shifted topological recursion correlators $\omega_{g,n}$ on the $(r,1)$ spectral curve given by $x = z^r$ and $y = z^{1-r}$ with the convention that $A_k = 0$ for all $k > r-1$, i.e.
	\begin{multline}\label{eq:shiftedTR}
		\omega_{g,1+n}(z_0,z_{[n]})
		\coloneqq 
		\Res_{z= 0} \sum_{\substack{Z \subseteq x^{\leftarrow}(z) \setminus \{z\} \\ |Z| \geq 1 }}
			\frac{dz_0}{(z-z_0) \prod_{z '\in Z} \left(y(z)-y(z')\right) dx(z) }
			\omega'_{g,1+|Z|;n}( z, Z;z_{[n]}) + \\
			+ \delta_{n,0} A_{2g}\frac{dz_0}{z_0^{2g}},
	\end{multline}
	are generating functions for the descendant integrals of the $\Theta^{r,1}$-classes. More precisely:
	\begin{equation}
		\omega_{g,n}(z_{[n]})
		=
		\left( \frac{1}{r} \right)^{2g-2+n}
		\sum_{\substack{ k_1,\dots,k_n \ge 0 \\ a_1,\dots,a_n \in [r-1] }}
		\int_{\overline{\mathcal{M}}_{g,n}}
			\Theta^{r,1}_{g,n}(a) \prod_{i=1}^n \psi_i^{k_i} d \xi_{k_i,a_i}(z_i)
	\end{equation}
	with $d \xi_{k,a}(z) = (rk+a)!^{(r)}\frac{dz}{z^{rk+a+1}}$.
\end{proposition}

\begin{proof}
	The loop equations of \cref{thm:loop:eq} for $s = 1$ can be expressed as 
	\begin{equation}
		\mc{E}_{g,n}^{(k)}(z;z_{[n]})
		=
		\delta_{n,0} \delta_{g,k/2} A_k \frac{dx^k}{x^{k}}
		+
		\bigO\biggl( \frac{dx^k}{x^{k-1}} \biggr),
	\end{equation}
	for any $1 \leq k \leq r$ after expanding in $\hbar$. Here, we adopt the convention that $A_r = 0$. Then, the shifted topological recursion formula \eqref{eq:shiftedTR} is obtained by applying \cite[Theorem~3.10]{BBKN25} and evaluating the terms corresponding to the shift explicitly. The statement about the descendant integrals is then a direct consequence of \cref{thm:gTR}.
\end{proof}

\subsubsection{The case $2 \leq s \leq r-2$ with $r ,s$ coprime}
In all remaining cases, the loop equations of \cref{thm:loop:eq} do not determine a unique solution (see \cref{sec:W:constraints} for further discussion). However, when $(r,s)$ are coprime, one can write a Bouchard--Eynard-style formula for $\omega_{g,1}$, assuming that all correlators $\omega_{g',n'}$ with $2g'-2+n' < 2g-1$ are known.

\begin{lemma}
	Let $(r,s)$ be coprime. Then the correlator $\omega_{g,1}$ of the generalized topological recursion on the $(r,s)$ spectral curve admits the following formula:
	\begin{equation}\label{eq:shiftedTRn=1}
		\omega_{g,1}(z_0) = 
		\Res_{z= 0} \sum_{\substack{Z \subseteq x^{\leftarrow}(z) \setminus \{z\} \\ |Z| \geq 1 }}
			\frac{dz_0}{(z-z_0)\prod_{z '\in Z} \left(y(z)-y(z')\right) dx(z) }
			\omega'_{g,1+|Z|;0}( z, Z;\emptyset)
		+  A_{2g}\frac{dz_0}{z_0^{s(2g-1)+1}},
	\end{equation}
	where we adopt the convention that $A_{k} = 0$ for all $k > r-s $.
\end{lemma}

\begin{proof}
	We can obtain a stronger version of the loop equations of \cref{thm:loop:eq} when $n = 0$ and $(r,s)$ are coprime. Indeed, notice that the right-hand side of \cref{eq:Easdet} is 
	\begin{equation}
		dx(z)^r \det\left( y \Id_{r\times r} - \frac{1}{\hbar} \mathcal{D} \right).
	\end{equation}
	In the Laplace expansion of the determinant, only the identity cycle and the $r$-cycle 
	\begin{equation}
		\left(
1,\,
\langle 1+s\rangle_r,\,
\langle 1+2s\rangle_r,\,
\ldots,\,
\langle 1+(r-1)s\rangle_r
\right),
	\end{equation}
	where $\langle m \rangle $ denotes the representative of $m $ modulo $r$ between $1$ and $r$, give non-zero contributions. The latter cycle contributes the product of the upper $(s+1)$-th diagonal and the lower $(r-s+1)$-th diagonal. All together we get
	\begin{equation}
		y^s \prod_{i=1}^{r-s} \left( y - \frac{1}{x} \frac{2i + s -r -1}{2(r-s)} \right)
		-
		\frac{1}{\hbar^r} \frac{1}{x^{r-s}}
		=
		\sum_{t=0}^{r-s} y^{s+t} A_{r-s-t}\frac{1}{x^{r-s-t}}
		-
		\frac{1}{\hbar^r} \frac{1}{x^{r-s}}.
	\end{equation}
	Then, from \cref{eq:Easdet} we can read off the loop equation for $n=0$ as 
	\begin{equation}
		\mc{E}_0^{(k)}(z;\emptyset;\hbar)
		=
		\begin{cases}
			 A_k \frac{dx^k}{x^{k}} & 1 \leq k \leq r-s, \\[1ex]
			0& r-s < k \leq r-1, \\
			\frac{-1}{\hbar^r} \frac{dx^r}{x^{r-s}} & k=r.
		\end{cases}
	\end{equation}
	Expanding in $\hbar$ and applying the proof of~\cite[Theorem~3.10]{BBKN25} to our setting yields the result. Note that the term $\frac{-1}{\hbar^r} \frac{dx^r}{x^{r-s}}$ vanishes in the computation, as it is  regular at $z=0$.
\end{proof}

\section{\texorpdfstring{$\mathcal{W}$}{W}-constraints and Airy structures}
\label{sec:W:constraints}
We turn the loop equations for the correlators obtained in the previous section into $\mathcal{W}$-constraints for the descendant potential $Z^{r,s}$ of the $\Theta^{r,s}$-classes. We also  compare these $\mathcal{W}$-constraints with the $(r,s)$-Airy structures studied in \cite{BBCCN24}.

\subsection{\texorpdfstring{$\mathcal{W}$}{W}-constraints for the descendant potential}
In this section, we recast the loop equations derived in the previous section as a set of $\mathcal{W}$-constraints for the descendant potential.

\subsubsection{Twist-field representations}
We are interested in the principal $\mathcal{W}$-algebra of $\mathfrak{gl}_r$ at the shifted level $k + r = 1$ (known as the self-dual level), which we denote by $\mathcal{W}(\mathfrak{gl}_r)$ for notational simplicity. The algebra $\mathcal{W}(\mathfrak{gl}_r)$ is strongly and freely generated, as a vertex algebra, by $r$ fields denoted $W^i(z)$ for $i \in [r]$. The field $W^i(z)$ has conformal weight $i$.

We use the convention
\begin{equation}
	W^i(z) = \sum_{k \in \mathbb{Z}} W^i_k \, z^{-i-k}
\end{equation}
for the mode expansion of the generating fields $W^i(z)$, for any $i \in [r]$. The modes $W^i_k$ together with the commutator $[\, ,\, ]$ form a non-linear Lie algebra, and we denote by $\mathcal{U}_r$ its universal enveloping algebra. The modes $W^i_k$ form a PBW basis for $\mathcal{U}_r$.

There is a natural exhaustive ascending filtration on $\mathcal{U}_r$, given by conformal weight. We denote by $F_n \mathcal{U}_r$ the subspace of elements of conformal weight $\leq n$, and introduce a parameter $\hbar$ using the Rees construction:
\begin{equation}
	\mathcal{U}_r^\hbar \coloneqq \bigoplus_{n \geq 0} \hbar^n F_n \mathcal{U}_r.
\end{equation}
This endows $\mathcal{U}_r^\hbar$ with a graded algebra structure in powers of $\hbar$. We denote by $W^{\hbar,i}_k \coloneqq \hbar^i W^i_k$ the homogenization of the modes in $\mathcal{U}_r^\hbar$.
See, for instance,~\cite[Section~2.6.4]{Bou26} or~\cite[Section~2.1.4]{BBKN25} for more details.

To obtain differential constraints, we aim to construct a representation of $\mathcal{U}_r^\hbar$ in terms of differential operators. We begin by considering the Weyl algebra in the variables $\bm{x} = \set{x_i}_{i \in \mathbb{Z}_{>0}}$, that is, the algebra of differential operators in the variables $x_1,x_2,\ldots$ with polynomial coefficients. Since the number of variables is infinite, we take a suitable completion of the Weyl algebra, denoted $\mathcal{D}_{\mathbb{Z}_{>0}}$. This algebra is filtered, and we promote it to a graded algebra by introducing a parameter $\hbar$.
More precisely, assign degree one to all variables $x_i$ and partial derivatives $\partial_i = \frac{\partial}{\partial x_i}$, and define the subspaces $F_n \mathcal{D}_{\mathbb{Z}_{>0}}$ to contain all operators in $\mathcal{D}_{\mathbb{Z}_{>0}}$ of degree $\le n$. This defines a filtration on $\mathcal{D}_{\mathbb{Z}_{>0}}$, called the Bernstein filtration. We then define the Rees--Weyl algebra as
\begin{equation}
	\mathcal{D}^\hbar_{\mathbb{Z}_{>0}}
	\coloneqq
	\bigoplus_{n \in \mathbb{Z}_{\geq 0}} \hbar^n F_n \mathcal{D}_{\mathbb{Z}_{>0}}.
\end{equation}
See \cite[Definition~2.4]{Bou26} for more details (see also \cite{BCJ24,BBKN25}). 
A typical element $P \in \mathcal{D}^\hbar_{\mathbb{Z}_{>0}}$ has the form
\begin{equation}
	P
	=
	\sum_{n \in \mathbb{Z}_{\geq 0}} \hbar^n
	\sum_{\substack{m,k \in \mathbb{Z}_{\geq 0} \\ m+k = n}}
	\sum_{i_1,\ldots,i_m \geq 1}
	p^{(k)}_{i_1,\ldots,i_m} \, \partial_{i_1} \cdots \partial_{i_m},
\end{equation} 
where $p^{(k)}_{i_1,\ldots,i_m}$ is a polynomial in the variables $\bm{x}$ of degree $\le k$ and only finitely many terms in the sum over $n$ are non-vanishing (i.e., the expression is polynomial in $\hbar$).

To formulate the $\mathcal{W}$-constraints for the descendant potential of the $\Theta^{r,s}$-classes, for each $(r,s)$ we construct a representation $\rho^s: \mathcal{U}_r^\hbar \to \mathcal{D}^\hbar_{\mathbb{Z}_{>0}}$ as follows. The vertex algebra $\mathcal{W}(\mathfrak{gl}_r)$ embeds into a Heisenberg algebra of rank $r$ via the quantum Miura transform, under which the generating field $W^i(z)$ for $i \in [r]$ is realized as the $i$-th elementary symmetric polynomial in the Heisenberg fields. In \cite{Mil16, BBCCN24}, certain representations of $\mathcal{W}(\mathfrak{gl}_r)$, referred to as twist-field representations, were constructed by restricting $\mathbb{Z}_r$-twisted Heisenberg representations to $\mathcal{W}(\mathfrak{gl}_r)$. Then, the twist-field representation of the vertex algebra $\mathcal{W}(\mathfrak{gl}_r)$, together with a further dilaton shift depending on an integer $s \geq 1$ as considered in~\cite[Section~4.1]{BBCCN24} (see also \cite[Section~2.2]{BBKN25}), gives rise to an induced representation on $\mathcal{U}^\hbar_r$, which we denote by $\rho^s \colon \mathcal{U}^\hbar_r \to \mathcal{D}^\hbar_{\mathbb{Z}_{>0}}$. An explicit expression for the modes 
\begin{equation}
	H^i_{k} \coloneqq \rho^s(W^{\hbar,i}_k)
\end{equation}
is also derived in~\cite{BBCCN24}. Of course, the operators $H^i_k$ depend on both $r$ and $s$, but we omit this dependence for ease of notation.

Before stating this expression, we need to set up some notation. Let $\theta$ be a primitive $ r$-th root of unity. Consider the following sums over roots of unity. Given $r \geq 2$ and $i \in [r]$ and $0 \leq j \leq \lfloor \frac{i}{2} \rfloor$, define
\begin{equation}
	\Psi^{(j)}_r (a_{2j+1}, \dotsc, a_i) \coloneqq \frac{1}{i!} \sum_{\substack{m_1, \dotsc, m_i = 0\\ m_l \neq m_k}}^{r-1} \prod_{k = 1}^j \frac{\theta^{m_{2k-1} + m_{2k}}}{(\theta^{m_{2k-1}} - \theta^{m_{2k}})^2} \prod_{l=2j+1}^i \theta^{-m_l a_l} \, .
\end{equation}
We also define, for any $m \in \mathbb Z$,
\begin{equation}
	J_m
	\coloneqq
	\begin{cases}
		 \partial_{m} & m \geq 1, \\
		0 & m = 0, \\
		(-m) x_{-m} - \frac{1}{\hbar} \delta_{m,-s} & m \leq - 1.
	\end{cases}
\end{equation}
Then, we have the following explicit expression of the modes of $\mathcal{W}(\mathfrak{gl}_r)$ in the representation $\rho^s$.

\begin{lemma}[{\cite[Corollary~4.7]{BBCCN24}}] \label{l:Wrep}
	Given $r \geq 2$ and $s \geq 1$, the operators $H^i_k = \rho^s(W^{\hbar,i}_k)$ take the form
	\begin{equation}\label{eq:Wmodes}
		H^{i}_k
		=
		\left(\frac{\hbar}{r}\right)^i
		\sum_{j=0}^{\lfloor \frac{i}{2} \rfloor}
			\frac{i!}{2^j j! (i-2j)!}
			\sum_{\substack{p_{2j+1}, \cdots p_i \in \Z \\ \sum p_l = rk}}
				\Psi^{(j)}_r (p_{2j+1}, \dotsc, p_i)
				:\mathrel{ \prod_{l=2j+1}^i J_{p_l} }: \,,
	\end{equation}
	where, for cases such that $j=i/2$, the condition $\sum p_l = rk$ is understood as $\delta_{k,0}$.
\end{lemma}

\begin{remark}
	Note that we have incorporated the dilaton shift $J_{-s} \mapsto J_{-s} - \frac{1}{\hbar}$ directly into the definition of the Heisenberg modes $J_m$, in contrast to the convention used in \cite{BBCCN24} (see also \cite[Section~2.2.5]{BBKN25}). The overall normalization of $H^i_k$ adopted here is more suitable for our purposes and agrees with \cite{BBKN25}, but differs from that of \cite{BBCCN24} by a factor of $r^{1-i}$. The $\hbar$-convention in this paper also differs from \cite{BBCCN24}, and aligns with the $\hbar^{2g-2+n}$ convention used in \eqref{eq:tau} (and followed in \cite{BCJ24,BBKN25,Bou26}). To obtain our $\hbar$-convention from that in \cite{BBCCN24}, the reader should first replace $\hbar$ with $\hbar^2$, and then rescale $x_m \to \hbar x_m$ for all $m \geq 1$ in the formulas of \cite{BBCCN24}.
\end{remark}

\subsubsection{\texorpdfstring{$\mathcal{W}$}{W}-constraints}
With this construction we are ready to recast the loop equations as $\mathcal{W}$-constraints. Consider the descendant potential of the $\Theta^{r,s}$-classes (this is the $r$-KdV tau function from \cref{thm:int}):
\begin{equation}\label{eq:Z}
	Z^{r,s}(\bm{x};\hbar)
	\coloneqq
	\exp \left(
	\sum_{\substack{g \geq 0, \, n \geq 1 \\ 2g-2+n>0}} \frac{\hbar^{2g-2+n}}{n!}
	\sum_{\substack{k_1,\ldots,k_n \geq 0 \\ a_1,\ldots,a_n \in [r-1]}}
	\int_{\Mbar_{g,n}}
	\Theta^{r,s}_{g,n}(a)
	\prod_{i=1}^n \psi_i^{k_i} (rk_i+a_i)!^{(r)} x_{rk_i + a_i}
	\right). 
\end{equation}
This function $Z^{r,s}$ satisfies the following $\mathcal{W}$-constraints.
 
\begin{theorem}[{$\mathcal{W}$-constraints}]\label{thm:W:const}
	Consider the modes $H^i_k$ of $\mathcal{W}(\mathfrak{gl}_r)$ in the representation \eqref{eq:Wmodes}. Then, for any $ r \geq 2 $ and $s \in [r-1]$, we have 
	\begin{equation}\label{eq:Wconst}
		H^i_k Z^{r,s} = \begin{cases}
		 \hbar^i A_i \delta_{k,0} Z^{r,s} & i \in [r-s] \, , k \geq 0 \\
		 0 & r-s +1\leq i \leq r\, , k \geq r-s-i+1,
		\end{cases}
	\end{equation} 
	where $A_i$ was defined in \eqref{eq:Akdef} (recall that $A_i$ vanishes when $i$ is odd).
\end{theorem}

\begin{proof}
	The proof essentially follows from the techniques used in \cite{BBCCN24} to study the $(r,s)$ Airy structures (and further developed in \cite{BKS24, BBCC24,BBKN25}). We give a sketch of the proof here for the reader's convenience. Consider the combination
	\begin{equation} \label{eq:ZHZ}
		\mathcal{H}^i(z; \bm{x} ;\hbar)
		\coloneqq
		Z^{r,s} (\bm{x};\hbar)^{-1} \left( \sum_{k \in \mathbb Z} H^i_k \frac{dx^i}{x^{i+k}} \right) Z^{r,s} (\bm{x}; \hbar),
	\end{equation}
	where $x=z^r$.
	Also consider the operator $\operatorname{ad}_n$ which picks terms that are homogeneous of degree $n$ in the variables $x_{\alpha_1},\ldots, x_{\alpha_n}$, and performs the substitution
	\begin{equation}
		x_{\alpha_i} \longmapsto \frac{dz_i}{z_i^{\alpha_i+1}}.
	\end{equation}
	Then the result of \cite[Sections~4.3--4.4]{BKS24} states that
	\begin{equation}
		\operatorname{ad}_n \mathcal{H}^i(z; \bm{x} ;\hbar) = \hbar^i \mathcal{E}^{(i)}_n(z; z_{[n]} ;\hbar).
	\end{equation}
	Now, the loop equations of \cref{thm:loop:eq} state that for all $i \in [r-s]$
	\begin{equation}
		\operatorname{ad}_n \mathcal{H}^i(z; \bm{x} ;\hbar)
		=
		\delta_{n,0} \hbar^i A_i \frac{dx^i}{x^i} + \bigO\biggl( \frac{dx^i}{x^{i-1}} \biggr).
	\end{equation} 
	As there are no terms of order $x^{-i-k}$ with $ k >0$, we see from \cref{eq:ZHZ} that $H^i_k$ for $i \in [r-s]$ must annihilate $Z^{r,s}$ for all $k > 0$. As for $ k =0$, the same logic gives $H^i_0 Z^{r,s} = \hbar^i A_i Z^{r,s} $. The case $r-s <i \leq r$ can be treated similarly using the corresponding loop equations in \cref{thm:loop:eq}.
\end{proof}

What we have shown is that the descendant potential of the $\Theta^{r,s}$-class satisfies $\mathcal{W}$-constraints, for all $r \geq 2$ and $s \in [r-1]$. It is interesting to ask further whether the $\mathcal{W}$-constraints uniquely fix the descendant potential. That is, for a given $(r,s)$, do the differential constraints of \cref{thm:W:const} have a unique solution of the form
\begin{equation}
	Z =
	\exp \left(
			\sum_{\substack{g \geq 0, \, n \geq 1 \\ 2g-2+n>0}} \frac{\hbar^{2g-2+n}}{n!}
			F_{g,n}
		\right),
\end{equation} 
for some homogeneous polynomials $F_{g,n}$ of degree $n$ in the variables $\bm{x}$? This is a question that can be answered within the framework of Airy structures. As we now show, the answer is yes only when $s=1$ or $s=r-1$.

\subsection{\texorpdfstring{$\mathcal{W}$}{W}-algebras and Airy structures}
To address the uniqueness question, we recall from \cite{BBCCN24} (see also \cite{BKS24,BBKN25}) the construction of the $(r,s)$ Airy structures.

\subsubsection{Airy structures}
Let us first recall the definition of Airy structures. These were introduced in \cite{KS18} as an algebraic reformulation (and generalization) of the Eynard--Orantin topological recursion \cite{EO07}. We provide only a brief overview here and refer the reader to the lecture notes \cite{Bou26} and the papers \cite{KS18,ABCO24,BBCCN24,BCJ24,BBKN25} for further details.

Let $A$ be a finite or countably infinite index set, and let $\mathcal{D}_A$ denote the Weyl algebra in the variables \smash{$\set{ x_a }_{a \in A}$}. Let $\mathcal{D}^\hbar_A$ be the Rees--Weyl algebra associated with the Bernstein filtration, as in the previous section. Airy structures (or Airy ideals) are particular left ideals in $\mathcal{D}^\hbar_A$.

\begin{definition}\label{d:airy}
	A left ideal $\mathcal{J} \subset \mathcal{D}^\hbar_A$ is called an \emph{Airy structure} (or \emph{Airy ideal}) if there exists a bounded\footnote{
		See \cite[Definition~2.15]{BCJ24} or \cite[Definition~2.6]{Bou26} for the definition of a set of bounded operators.
	} generating set \smash{$\set{ H_a }_{a\in A}$} for $\mathcal{J}$ such that:
	\begin{enumerate}
		\item\label{airy:hbar}
		The operators $H_a$ take the form 
		\begin{equation}
			H_a = \hbar \partial_a + \bigO(\hbar^2).
		\end{equation}

		\item
		The left ideal $\mathcal{J}$ satisfies $[\mathcal{J}, \mathcal{J}] \subseteq \hbar^2 \mathcal{J}$.
	\end{enumerate}
\end{definition}

The principal motivation for studying Airy structures is the following foundational theorem from \cite{KS18}.

\begin{theorem}[{\cite{KS18}}]
	Let $\mathcal{J} \subset \mathcal{D}^\hbar_A$ be an Airy structure. Then there exists a unique function $Z$ of the form 
	\begin{equation}
		Z
		=
		\exp \left(
			\sum_{\substack{g \in \frac{1}{2} \mathbb{Z}_{\geq 0}, \, n \in \mathbb{Z}_{>0} \\ 2g-2+n>0}}
				\frac{\hbar^{2g-2+n}}{n!}
				F_{g,n}
		\right),
	\end{equation}
	where each $F_{g,n}$ is a homogeneous polynomial of degree $n$ in the variables \smash{$\set{ x_a }_{a \in A}$}, such that $\mathcal{J}$ is the annihilator ideal of $Z$ in $\mathcal{D}^\hbar_A$. That is, $Z$ is the unique solution to the differential constraints
	\begin{equation}
		H_a Z = 0
		\qquad
		\forall a \in A,
	\end{equation}
	of the above form. The function $Z$ is called the \emph{partition function} of the Airy structure $\mathcal{J}$.
\end{theorem}

\subsubsection{The $(r,s)$ Airy structures}
The notion of Airy structures is relevant because, if we can show that the constraints from \cref{thm:W:const} form an Airy structure, then we can conclude that they uniquely determine the descendant potential. If they do not form an Airy structure, then a more detailed analysis of the constraints is required.

An interesting class of Airy structures to compare with was constructed in \cite{BBCCN24}, and will be referred to as the \emph{$(r,s)$ Airy structures}; see also \cite[Section~2.2]{BBKN25}. These are based on the representation of $\mathcal{W}(\mathfrak{gl}_r)$ at self-dual level constructed in \cref{l:Wrep}.

\begin{theorem}[{\cite[Theorem~4.9]{BBCCN24}}]\label{t:airy}
	Let $r \geq 2$ and $s \in [r+1]$ such that $r \equiv \pm 1 \pmod{s}$. Consider the representation and the associated operators
	\begin{equation}
		\rho^s\colon \mathcal{U}^\hbar_r \longrightarrow \mathcal{D}^\hbar_{\mathbb{Z}_{>0}},
		\qquad\qquad
		H^i_k = \rho^s(W^{\hbar,i}_k)
	\end{equation}
	as in \cref{l:Wrep} and \cref{eq:Wmodes}. Let $\mathcal{J} \subset \mathcal{D}^\hbar_{\mathbb{Z}_{>0}}$ be the left ideal generated by the operators
	\begin{equation}\label{eq:rsmodes}
		\Set{
			H^i_k | i \in [r], \ k \geq -\left\lfloor \frac{s(i-1)}{r} \right\rfloor
		}.
	\end{equation}
	Then $\mathcal{J}$ is an Airy structure, which we call the \emph{$(r,s)$ Airy structure}.
\end{theorem}

\begin{remark}[Admissibility]
	Note that these Airy structures are defined only when the admissibility condition $r \equiv \pm 1 \pmod{s}$ is satisfied. Moreover, \cite{BBCCN24} proves that the $F_{g,n}$ of the associated partition function reconstruct the correlators $\omega_{g,n}$ computed by the Bouchard--Eynard topological recursion~\cite{BE13,BHLMR14} on the $(r,s)$ spectral curves $x = z^r$, $y = z^{s-r}$. In fact, \cite{BBCCN24} also proves that the Bouchard--Eynard topological recursion is well-defined—meaning that the correlators $\omega_{g,n}$ produced by the recursive formulas are symmetric—if and only if the admissibility condition $r \equiv \pm 1 \pmod{s}$ is satisfied.
\end{remark}

\subsubsection{The shifted $(r,s)$ Airy structures}
A slightly larger class of Airy structures can be obtained as representations of $\mathcal{W}(\mathfrak{gl}_r)$ at self-dual level; those were constructed in~\cite[Section~2.3]{BBKN25}, and will be referred to as \emph{shifted $(r,s)$ Airy structures}.

The idea is simple: we start with one of the $(r,s)$ Airy structures, and we introduce ``highest weights'', that is, we shift the zero modes $H^i_0$ by terms of the form $\sum_{n=2}^\infty \hbar^n S_{i,n}$ for some constants\footnote{
	To be precise, we could start at $n=1$ here as in~\cite{BBKN25}, but to do this we would need to extend the definition of Airy structures slightly to allow $\bigO(\hbar)$ terms in the generators $H_a$, as done in~\cite[Definition~2.3]{BBKN25}. As this is not needed in this paper, we avoid this unnecessary complication.
} $S_{i,n} \in \mathbb{C}$. However, after doing this, we need to make sure that the conditions in the definition of Airy structures are still satisfied. The first condition on the form of the other operators is obviously still satisfied, but the second condition on the ideal, that is $[\mathcal{J},\mathcal{J}] \subseteq \hbar^2 \mathcal{J}$, is highly non-trivial. The question of when it remains satisfied is studied in detail in~\cite[Section~2.3]{BBKN25}. The result is the following.

\begin{definition}[{\cite[Definition~2.26]{BBKN25}}]\label{d:sconsistent}
	Let $S = \Set{S_{i,n}}_{i \in [r], n \geq 2}$ be a set of complex numbers. We say that it is \emph{$s$-consistent} if the following three conditions are satisfied:
	\begin{itemize}
		\item If $s = 1$, no condition is imposed;

		\item If $s \geq 2$ and $r \equiv 1 \pmod{s}$, then $S_{i,n}=0$ for all $2 \leq i \leq r$;

		\item If $s \geq 3$ and $r \equiv -1 \pmod{s}$, then $S_{i,n}=0$ for all $i \in [r]$.
	\end{itemize}
\end{definition}

We then obtain the following new class of shifted $(r,s)$ Airy structures:

\begin{theorem}[{\cite[Theorem~2.27]{BBKN25}}]\label{t:shiftedairy}
	Let $r \geq 2$ and $s \in[r+1]$ such that $r \equiv \pm 1 \pmod{s}$. Consider the representation $\rho^s: \mathcal{U}^\hbar_r \to \mathcal{D}^\hbar_{\mathbb{Z}_{>0}}$ from \cref{l:Wrep}, and define the shifted operators 
	\begin{equation}
		G^i_k = \rho^s(W^{\hbar,i}_k) - \delta_{k,0} \sum_{n=2}^\infty \hbar^n S_{i, n},
	\end{equation}
	where $S = \Set{S_{i,n}}_{i \in [r], n \geq 2}$ is $s$-consistent. Let $\mathcal{J} \subset \mathcal{D}^\hbar_{\mathbb{Z}_{>0}}$ be the left ideal generated by the operators
	\begin{equation}\label{eq:rsmodes:shifted}
		\Set{
			G^i_k | i \in [r], \ k \geq -\left\lfloor \frac{s(i-1)}{r} \right\rfloor
		}.
	\end{equation}
	Then $\mathcal{J}$ is an Airy structure, which we call a \emph{shifted $(r,s)$ Airy structure}.
\end{theorem}

Note that the $s$-consistency condition is quite stringent. It can be summarized as follows:
\begin{itemize}
	\item
	For $s = 1$, all zero modes are shifted:
	\begin{equation}
		G^i_0 = \rho^s(W^{\hbar,i}_0) - \sum_{n=2}^\infty \hbar^n S_{i, n}.
	\end{equation}
	
	\item
	For $s \geq 2$ and $r \equiv 1 \pmod{s}$, only the first zero mode is shifted:
	\begin{equation}
		G^i_0 = \rho^s(W^{\hbar,i}_0) - \delta_{i,1} \sum_{n=2}^\infty \hbar^n S_{i, n}.
	\end{equation}
	
	\item
	For $s \geq 3$ and $r \equiv -1 \pmod{s}$, no zero modes can be shifted at all.
\end{itemize}

\subsection{Uniqueness}
With this background on Airy structures under our belt, we can answer the question whether the $\mathcal{W}$-constraints derived in \cref{thm:W:const} uniquely fix the descendant potential. We can analyze three cases separately: the case $s = r-1$, which retrieves the result of \cite{CGG25}, the case $s = 1$, which finds an enumerative-geometric interpretation of the $\mathcal{W}$-constraints found in both \cite{YZ23} and \cite{BBKN25}, and the remaining case.

\subsubsection{The case $s=r-1$}
This case was already studied in \cite{CGG25}.
 
\begin{proposition}[{\cite[Theorem~5.5]{CGG25}}]\label{prop:s-1}
	The $\mathcal{W}$-constraints satisfied by the descendant potential of the $\Theta^{r,r-1}$-classes from \cref{thm:W:const} (the case $s=r-1$) can be rewritten as
	\begin{equation}
		H^i_k Z^{r,r-1} = 0, 
		\qquad
		i \in [r],
		\quad
		k \geq 2-i - \delta_{i,1}.
	\end{equation}
	Then the left ideal $\mathcal{J} \subset \mathcal{D}^\hbar_{\mathbb{Z}_{>0}}$ generated by these $H^i_k$ forms an Airy structure, namely the $(r,r-1)$ Airy structure of \cref{t:airy}. As a result, the $\mathcal{W}$-constraints uniquely fix the descendant potential.
\end{proposition}

\begin{proof}
	When $s=r-1$, \cref{thm:W:const} reduces to the constraints
	\begin{equation}\label{eq:Wconstr-1}
		H^i_k Z^{r,r-1} = \begin{cases}
		 \hbar^i A_i \delta_{k,0} Z^{r,r-1} & i =1 \, , k \geq 0 \\
		 0 & 2 \leq i \leq r\, , k \geq 2 -i.
		\end{cases}
	\end{equation} 
	However, the constants $A_i$ vanish whenever $i$ is odd. Thus $A_1=0$, and the constraints above become simply
	\begin{equation}\label{eq:r-1const}
		H^i_k Z^{r,r-1} = 0, \qquad i \in[r],\quad k \geq 2-i - \delta_{i,1}.
	\end{equation}

	We can compare with the $(r,r-1)$ Airy structure of \cref{t:airy}. When $s=r-1$, for $i \in [r]$,
	\begin{equation}
		\left\lfloor \frac{s(i-1)}{r} \right\rfloor =\left\lfloor \frac{(r-1)(i-1)}{r} \right\rfloor = i-2 + \delta_{i,1},
	\end{equation}
	and the constraints from \cref{t:airy} are the same as \eqref{eq:r-1const}, which is the statement of the proposition.
\end{proof}
 
\subsubsection{The case $s=1$}
We obtain the following result, which is new:

\begin{proposition}\label{prop:s=1}
	The $\mathcal{W}$-constraints satisfied by the descendant potential of the $\Theta^{r,1}$-classes from \cref{thm:W:const} (the case $s=1$) can be rewritten as
	\begin{equation}\label{eq:Wconst1bis}
		(H^i_k-\hbar^i A_i \delta_{k,0}) Z^{r,1} = 0,
		\qquad
		i \in [r],
		\quad
		k \geq 0,
	\end{equation} 
	with the constants $A_i$ defined by (cf. \eqref{eq:Akdef})
	\begin{equation}\label{eq:Aicons}
		A_i
		\coloneqq
		\begin{cases}
			e_i\left(\frac{2-r}{2(r-1)}, \frac{4-r}{2(r-1)}, \ldots, \frac{2(r-1)-r}{2(r-1)}\right)
			& i \in[r-1],\\
			0
			& i=r,
		\end{cases}
	\end{equation}
	with $e_i$ the $i$-th elementary symmetric polynomial. Then the left ideal $\mathcal{J} \subset \mathcal{D}^\hbar_{\mathbb{Z}_{>0}}$ generated by these $H^i_k - \hbar^i A_i \delta_{k,0}$ forms an Airy structure, namely the shifted $(r,1)$ Airy structure of \cref{t:shiftedairy} with the shifts given by
	\begin{equation}
		S_{i,n} = \delta_{i,n} A_i,
		\qquad
		i \in [r].
	\end{equation}
	As a result, the $\mathcal{W}$-constraints uniquely fix the descendant potential.
\end{proposition}
 
\begin{proof}
	The constraints from \cref{thm:W:const} reduce to
	\begin{equation}\label{eq:Wconst1}
		H^i_k Z^{r,1}
		=
		\hbar^i A_i \delta_{k,0} Z^{r,1},
		\qquad
		i \in [r],
		\quad
		k \geq 0,
	\end{equation} 
	with the constants $A_i$ given by \eqref{eq:Aicons}. In other words, the constraints \eqref{eq:Wconst1} only involve non-negative modes, with the zero modes $H^i_0$ acting as a constant $\hbar^i A_i$ for $i \in [r-1]$ and the highest zero mode $H^r_0$ acting as zero. One can then think of $Z^{r,1}$ as a highest weight state for this particular choice of highest weight.

	We can compare with the shifted $(r,1)$ Airy structure of \cref{t:shiftedairy}. When $s=1$, for $i \in [r]$,
	\begin{equation}
		\left\lfloor \frac{s(i-1)}{r} \right\rfloor
		=
		\left\lfloor \frac{i-1}{r} \right\rfloor
		=
		0.
	\end{equation}
	Thus the shifted $(r,1)$ Airy structure of \cref{t:shiftedairy} is the left ideal generated by the operators
	\begin{equation}\label{eq:r1modes}
		\Set{
			H^i_k - \delta_{k,0} \sum_{n=2}^\infty \hbar^n S_{i,n} | i \in [r], \ k \geq 0
		}.
	\end{equation}
	We conclude that if we set the weights to (recall that $A_1 = 0$):
	\begin{equation}
		S_{i,n} = \delta_{n,i} A_i, \qquad i \in [r],
	\end{equation}
	then we recover the operators from \eqref{eq:Wconst1bis}. We conclude that the constraints from \cref{thm:W:const} when $s=1$ form an Airy structure, namely the shifted $(r,1)$ Airy structure from \cref{t:shiftedairy} with the shifts given by $S_{i,n} = \delta_{n,i} A_i$, $i\in [r]$.
\end{proof}

As explained in \cref{rem:YZ}, the solution to the $r$-KdV hierarchy defined by the initial condition
\begin{equation}\label{eq:rKdV:ic}
	u_{\alpha}(x_1;\hbar) = d_{\alpha} \left( \frac{\hbar}{1 - \hbar x_1} \right)^{\alpha+1},
	\qquad
	\alpha \in [r-1],
\end{equation}
was studied in detail in \cite{YZ23}. In \textit{loc. cit.}, the authors defined the \emph{generalized BGW tau function}, denoted $\tau_{\textup{BGW}}(\bm{x};\hbar;d_1,\ldots,d_{r-1})$, as the unique Dubrovin--Zhang type tau function associated with the solution of the $r$-KdV hierarchy satisfying the initial condition \eqref{eq:rKdV:ic}, the normalization $\tau_{\textup{BGW}} = 1 + \bigO(\hbar)$, and the so-called string equation\footnote{As our convention for the definition of $u_\alpha$ in terms of $Z^{r,s}$ given in \eqref{eq:ualpha:def} differs from the one in \cite{YZ23} by a factor of $\frac{1}{r}$, our string equation appears with the constant term $d_1$ as opposed to $\frac{d_1}{r}$.}:
\begin{equation}\label{eq:string}
	\left(
		\hbar\partial_{1}
		-
		\hbar^2 \sum_{k \ge 0,\, a \in [r-1]} (rk + a) x_{rk+a} \partial_{rk+a}
		-
		\hbar^2 d_1
	\right) \tau_{\textup{BGW}}
	=
	0.
\end{equation}
For a specific choice of constants $d_\alpha$, we prove that the generalized BGW tau function of Yang--Zhou coincides with the descendant potential of the $\Theta^{r,1}$-classes, thereby giving the former an enumerative interpretation.

\begin{proposition}\label{prop:YZ}
	The descendant potential of the $\Theta^{r,1}$-classes equals the generalized BGW tau function of Yang--Zhou corresponding to the choice of constants
	\begin{equation}\label{eq:YZdalpha}
		d_{\alpha} = \delta_{\alpha,2g-1} \, \alpha^2 \int_{\Mbar_{g,1}} \Theta^{r,1}_{g,1}(\alpha).
	\end{equation}
\end{proposition}

\begin{proof}
	The first part of \cref{prop:ic} shows that $Z^{r,1}$ is an $r$-KdV tau function with the initial condition
	\begin{equation}
		u_\alpha(x_1;\hbar)
		=
		\delta_{\alpha,2g-1}
		\left( \alpha^2 \int_{\Mbar_{g,1}} \Theta^{r,1}_{g,1}(\alpha) \right)
		\left( \frac{\hbar}{1 -\hbar x_1} \right)^{\alpha+1} .
	\end{equation}
	Thus, all we need to show is that $Z^{r,1}$ also satisfies the string equation~\eqref{eq:string} with the choice of $d_1 $  dictated by \eqref{eq:YZdalpha}. We can explicitly compute this intersection number directly from Chiodo's formula (see \cite[Example~2.2.21]{Gia21}) or from any of the formulas of \cref{sec:gen:TR:det:formulas}:
	\begin{equation}
		d_1
		=
		\int_{\Mbar_{1,1}} \Theta^{r,1}_{1,1}(1)
		=
		\frac{r}{1-r}
		\int_{\Mbar_{1,1}} C_{1,1}^{r,1}(1)
		=
		\frac{1}{24} \frac{(r-1)r + 1}{(r-1)} .
	\end{equation}
	Now, using the explicit expression for $\Psi^{(i)}_r$ from~\cite[Lemma~A.5]{BBCCN24} and the fact that $\partial_{kr}Z^{r,1} = 0$, we find that the constraint corresponding to $H^2_0$ from \cref{prop:s=1} reads as
	\begin{equation}\label{eq:H20}
		r H^{2}_0 Z^{r,1}
		=
		\left(
		\hbar \partial_1
		-
		\hbar^2
		\sum_{k \ge 0,\, a \in [r-1]} (rk + a) x_{rk+a} \partial_{rk+a}
		-
		\hbar^2 \,
		\frac{r^2-1}{24}
		\right)
		Z^{r,1} = \hbar^2 r A_2 Z^{r,1}.
	\end{equation}
	We can compute \smash{$A_2 = - \frac{1}{24}\frac{(r-2)r}{r-1}$} using the recursive definition of the elementary symmetric polynomials, and thus the equation above matches the string equation~\eqref{eq:string}. Hence, $Z^{r,1}$ satisfies the string equation as well. The uniqueness of the tau function then yields the statement of the proposition.
\end{proof}

\subsubsection{The remaining cases}
We now consider the remaining cases, namely $2 \leq s \leq r-2$ and $r \geq 4$. In this range, the $\mathcal{W}$-constraints from \cref{thm:W:const} do not coincide with either the $(r,s)$ Airy structures of \cref{t:airy} or the shifted $(r,s)$ Airy structures of \cref{t:shiftedairy}. This is straightforward to see: whenever $r \geq 4$ and $s \leq r-2$, at least one of the shifts---specifically $A_2$---in the $\mathcal{W}$-constraints is non-zero, and hence they cannot coincide with the unshifted $(r,s)$ Airy structures of \cref{t:airy}. Therefore, they could only potentially match the shifted $(r,s)$ Airy structures of \cref{t:shiftedairy}. But since $A_2 \neq 0$, the only possibility compatible with the definition of $s$-consistent shifts in \cref{d:sconsistent} is $s=1$, which we have already treated.

In fact, as we shall now see, for these values of $(r,s)$ the $\mathcal{W}$-constraints of \cref{thm:W:const} do not define an Airy structure at all. To verify this, we examine the $\bigO(1)$ and $\bigO(\hbar)$ terms in the operators $H^i_k$ from \cref{thm:W:const} to determine whether condition~(\labelcref{airy:hbar})  in the definition of an Airy structure (\cref{d:airy})  is satisfied. Let us first formalize the relevant index set from \cref{thm:W:const}.

\begin{definition}\label{d:indexset}
	Let $I_r \coloneqq \Set{ (i,k) | i \in [r],\ k \in \mathbb{Z} }$. We define the subset $I_{r,s} \subset I_r$ by the conditions:
	\begin{itemize}
		\item for $i \in [r-s]$, we require $k \geq 0$;

		\item for $r-s+1 \leq i \leq r$, we require $k \geq r-s-i+1$.
	\end{itemize}
	With this notation, the operators in \cref{thm:W:const} are the $H^i_k$ for $(i,k) \in I_{r,s}$.
\end{definition}

The $\bigO(1)$ term in the $\hbar$-expansion of the operator $H^i_k$ from \eqref{eq:Wmodes} has the form (originating from the dilaton shift in $J_{-s}$):
\begin{equation}
	H^i_k = \mu^i_{r,s} \, \delta_{k,-\frac{is}{r}} + \bigO(\hbar),
\end{equation}
where $\mu^i_{r,s}$ is a non-zero but irrelevant constant. However, it is easy to see that $\left(i,-\frac{is}{r}\right) \notin I_{r,s}$ for all $i \in [r]$, so none of the operators in \cref{thm:W:const} contain $\bigO(1)$ terms.

The $\bigO(\hbar)$ term of $H^i_k$ also arises from the dilaton shift and takes the form
\begin{equation}
	H^i_k = \nu^i_{r,s} \, \hbar J_{(i-1)s + rk} + \bigO(\hbar^2),
\end{equation}
where again $\nu^i_{r,s}$ is an irrelevant non-zero constant. In order to satisfy condition~(\labelcref{airy:hbar}) of \cref{d:airy}, we would need all modes $J_m$ with $m \in \mathbb{Z}_{>0}$ to appear exactly once as the $\bigO(\hbar)$ term of some $H^i_k$. To analyze whether this condition is met, we introduce the following map.

\begin{definition}\label{d:rsmap}
	Let $I_r \coloneqq \Set{ (i,k) | i \in [r],\ k \in \mathbb{Z} }$ as above. We define the map
	\begin{equation}
		\Pi_{r,s} \colon I_r \longrightarrow \mathbb{Z},
		\qquad
		(i,k) \longmapsto (i-1)s + r k.
	\end{equation}
\end{definition}

Condition~(\labelcref{airy:hbar}) in \cref{d:airy} is then equivalent to the requirement that the restriction of $\Pi_{r,s}$ to $I_{r,s} \setminus \set{0} \subset I_r$ is a bijection onto $\mathbb{Z}_{> 0}$.
 
\begin{lemma}\label{l:restriction}
	Let $K_{r,s} \coloneqq \Set{ m \in \mathbb{Z}_{>0} | m \neq a (r-s)+ b s \text{ for any $a,b \in \mathbb{Z}_{\geq 0}$} }$. Note that $K_{r,s}$ is finite if and only if $r$ and $s$ are coprime, and $K_{r,s} = \emptyset$ if and only if $s=1$ or $s=r-1$. Then, for all $r \geq 2$ and $s \in [r-1]$,
	\begin{equation}
		\Pi_{r,s}(I_{r,s}) =  \set{0} \cup\bigl(\mathbb Z_{>0}\setminus K_{r,s}\bigr) ,
	\end{equation}
	and $\Pi_{r,s}\big|_{I_{r,s}}$ is injective if and only if $r$ and $s$ are coprime.
\end{lemma}

\begin{proof}
	First, for the initial remarks. Let $g = \gcd(r,s)$. Clearly, all positive integers that are not multiples of $g$ are in $K_{r,s}$. Therefore $K_{r,s}$ is infinite if $g \geq 2$. As for the coprime case $g=1$, it is a classical result (the Frobenius coin-exchange problem) that the largest positive integer that cannot be written as $m=a(r-s)+bs$ for some $a,b \in \mathbb{Z}_{\geq 0}$ is $s(r-s)-r$. Consequently, $K_{r,s}$ is finite. Moreover, $K_{r,s}$ is empty when $s=1$ or $s=r-1$, while if $2 \leq s \leq r-2$, then $1 \in K_{r,s}$ and hence it is non-empty.

	Now to the main part of the lemma. Write $I_{r,s} = \set{(1,0)} \cup I^1_{r,s} \cup I^2_{r,s}$, where
	\begin{equation}
	\begin{aligned}
		I^1_{r,s}
		&=
		\Set{ (i,k) | i \in [r-s], \ k \geq 0 , (i,k) \neq (1,0)}, \\
		I^2_{r,s}
		&=
		\Set{ (i,k) | r-s+1\leq i \leq r, \ k \geq r-s-i+1 }.
	\end{aligned}
	\end{equation}
	Simple index manipulations show that:
	\begin{equation}
	\begin{aligned}
		\Pi_{r,s}(I^1_{r,s})
		&=
		\Set{ m \in \mathbb{Z}_{>0} | m = k (r-s)+(i-1+k) s, \ i \in [r-s], \ k \geq 0 },\\
		\Pi_{r,s}(I^2_{r,s})
		&=
		\Set{ m \in \mathbb{Z}_{>0} | m = (i+k) (r-s)+ k s, \ i \in [s], \ k \geq 0 }.
	\end{aligned}
	\end{equation}
	This gives us a description of $\Pi_{r,s}(I_{r,s}) = \set{0} \cup \Pi_{r,s}(I^1_{r,s}) \cup \Pi_{r,s}(I^2_{r,s})$. 

	Now consider the set $J_{r,s} = \mathbb{Z}_{>0} \setminus K_{r,s}$ of positive integers that can be written as $a (r-s) +b s$ for some $a,b \in \mathbb{Z}_{\geq 0}$. Clearly, $\Pi_{r,s}(I_{r,s})\subseteq \set{0} \cup J_{r,s}$. Let us prove the other inclusion. Set $g = \gcd(r,s)$ as above, let $m \in J_{r,s}$, and pick a pair $a,b \in \mathbb{Z}_{\geq 0}$ such that $m = a(r-s)+bs$. Any other pair $a', b' \in \mathbb{Z}_{\geq 0}$ such that $m = a'(r-s) + b' s$ will satisfy
	\begin{equation}
		a' = a - \frac{k s}{g},
		\qquad
		b' = b + \frac{k (r-s)}{g}
		\qquad
		\text{for some }k \in \mathbb{Z}.
	\end{equation}
	In particular, $b'-a' = b-a + \frac{k r}{g}$. This means that for any $m \in J_{r,s}$, we can find a unique pair $A,B \in \mathbb{Z}_{\geq 0}$ such that $m = A (r-s) + B s$ and $-\frac{s}{g} \leq B-A \leq \frac{r-s}{g} -1$. When $B-A \geq 0$, $m \in \Pi_{r,s}(I^1_{r,s})$, while when $B-A < 0$, $m \in \Pi_{r,s}(I^2_{r,s})$. This shows that $J_{r,s} \subseteq   \Pi_{r,s}(I_{r,s})$. Moreover, $0$ has the unique preimage $(1,0)$, and we conclude that $J_{r,s}  \cup \set{0}  =  \Pi_{r,s}(I_{r,s})$. 

	Since the representation of $m \in J_{r,s}$ in terms of $A$ and $B$ is unique,  the restriction $\Pi_{r,s}\big|_{I_{r,s}}$ is  injective if $g=1$, that is, if  $r$ and $s$ are coprime. On the other hand, if $g>1$, for sufficiently large $k$, both $(1,k)$ and $(1+r/g,k-s/g)$ belong to  $I_{r,s}$ and map to the same point under $\Pi_{r,s}$. This concludes the proof.
\end{proof}

As a result of \cref{l:restriction}, we can immediately conclude that:

\begin{corollary}\label{c:notairy}
	For $r \geq 4$ and $2 \leq s \leq r-2$, the $\mathcal{W}$-constraints satisfied by the descendant potential of the $\Theta^{r,s}$-classes from \cref{thm:W:const} do not form an Airy structure.
\end{corollary}

However, all is not lost. The $\mathcal{W}$-constraints still fix a large part of the descendant potential.

\begin{definition}\label{d:reduced}
	Let $Z^{r,s}$ be the descendant potential of the $\Theta^{r,s}$-classes. We define the \emph{reduced descendant potential} $\hat{Z}^{r,s}$ to be:
	\begin{equation}\label{eq:Z:reduced}
		\hat{Z}^{r,s}(\bm{x};\hbar)
		\coloneqq
		\exp \left(
			\sum_{\substack{g \geq 0, \, n \geq 1 \\ 2g-2+n>0}} \frac{\hbar^{2g-2+n}}{n!}
			\sum_{\substack{k_1, \ldots, k_n \geq 0 \\ a_1, \ldots, a_n \in [r-1] \\ r k_i + a_i \in K_{r,s}}}
			\int_{\Mbar_{g,n}}
				\Theta^{r,s}_{g,n}(a)
				\prod_{i=1}^n \psi_i^{k_i} (rk_i + a_i)!^{(r)} x_{rk_i + a_i}
		\right). 
	\end{equation}
	In other words, we single out all polynomials in the variables $x_m$ with $m \in K_{r,s}$.
\end{definition}

\begin{proposition}\label{p:reduced}
	For any $r \geq 2$ and $s \in [r-1]$, given the reduced descendant potential $\hat{Z}^{r,s}$ of the $\Theta^{r,s}$-classes, the $\mathcal{W}$-constraints of \cref{thm:W:const} uniquely fix the descendant potential $Z^{r,s}$.
\end{proposition}

\begin{proof}
	Write the descendant potential as
	\begin{equation}
		Z^{r,s}
		=
		\exp \left(
			\sum_{\substack{g \geq 0, n \geq 1 \\ 2g-2+n>0}}
				\frac{\hbar^{2g-2+n}}{n!}
				\sum_{m_1,\ldots,m_n \in \mathbb{Z}_{>0}}
					F_{g,n}[m] \prod_{i=1}^n x_{m_i}
		\right).
	\end{equation}
	The operators in \cref{thm:W:const} take the form
	\begin{equation}
		C_\ell \, \hbar \partial_\ell + \bigO(\hbar^2)
	\end{equation}
	for all $\ell \in \mathbb{Z}_{>0} \setminus K_{r,s}$, where the $C_\ell$ are irrelevant non-zero constants. As a result, acting on the descendant potential $Z^{r,s}$, we get a recursion of the form
	\begin{equation}
		F_{g,n}[\ell, m_1,\ldots,m_{n-1}] = G_{g,n}[\ell,m_1,\ldots,m_{n-1}],
	\end{equation}
	where $G_{g,n}[\ell,m_1,\ldots,m_{n-1}]$ is a combination of $F_{g',n'}$ with $2g'-2+n' < 2g-2+n$. We get such an equation for all $\ell \in \mathbb{Z}_{>0} \setminus K_{r,s}$ and $m_1,\ldots,m_{n-1} \in \mathbb{Z}_{>0}$. By symmetry, this means that the constraints uniquely fix all $F_{g,n}[m_1,\ldots,m_n]$ with at least one entry in $\mathbb{Z}_{>0} \setminus K_{r,s}$ recursively in terms of $F_{g',n'}[m_1,\ldots,m_{n'}]$ with $m_1,\ldots,m_{n'} \in \mathbb{Z}_{>0}$ and $2g'-2+n'<2g-2+n$. 

	What remains unfixed by the constraints is the reduced descendant potential, i.e. the $F_{g,n}[m]$ with $m = (m_1,\ldots,m_n) \in K_{r,s}^n$. In other words, given the reduced descendant potential, the $\mathcal{W}$-constraints uniquely fix recursively the full descendant potential.
\end{proof}

\printbibliography

\end{document}